\documentclass{article}
\usepackage{amsthm}
\usepackage{amssymb}
\usepackage{amsmath}
\usepackage{float}
 \usepackage{mathtools}
\usepackage{comment}

\theoremstyle{plain}
\newtheorem{theorem}{Theorem}[section]
\newtheorem{coro}[theorem]{Corollary}
\newtheorem{lemma}[theorem]{Lemma}
\newtheorem{proposition}[theorem]{Proposition}

\theoremstyle{definition}
\newtheorem{definition}[theorem]{Definition}

\theoremstyle{remark}
\newtheorem{remark} [theorem]{Remark}

\renewcommand{\L}{\mathcal{L}}

\newcommand{\Ii}{\mathcal{I}}

\newcommand{\R}{\mathbb{R}}

\newcommand{\Z}{\mathbb{Z}}

\newcommand{\lieg}{\mathfrak{g}}
\renewcommand{\ss}{\mathfrak{s}}

\newcommand{\liea}{\mathfrak{a}}

\newcommand{\liej}{\mathfrak{j}}
\newcommand{\lieq}{\mathfrak{q}}
\newcommand{\lieh}{\mathfrak{h}}
\newcommand{\liek}{\mathfrak{k}}

\newcommand{\N}{\mathbb{N}}

\DeclareMathOperator{\Aut}{Aut}
\DeclareMathOperator{\End}{End}

\title{Some Anosov actions which are affine}
\author{Uir\'a N. Matos de Almeida
\thanks{The Author thanks CAPES, CNPq and FAPESP for the financial support, my advisors Thierry Barbot and Carlos Maqueraand my posdoc SUpervisor Pedro Salomão}}

\begin{document}

\maketitle

\begin{abstract}
Following the works of Y. Benoist, P. Foulon and F. Labourie \cite{BFL}, and having in mind the standing conjecture about the algebricity of Anosov actions of $\R^k$, we propose some geometrical conditions which generalizes the notion of contact structures, and prove that Anosov actions associated with such structures are conjugated to an Affine action. We also construct two families of examples, the first one is algebraic in nature, and the second one imposes some dynamical conditions instead of geometrical ones.

\end{abstract}

\tableofcontents 

\section{Introduction}
One of the classical problems in Anosov dynamics is to understand when is an Anosov system conjugate to, up to finite covers and reparametrizations, an algebraic model. It is, in fact, conjectured that this should be always the case for Anosov diffeomorphisms (Smale \cite{Smale}) and Anosov action of $\R^k$, $k\geq2$ (B. Kalinin and R. Spatzier \cite{KS}). It is well known that this is not always the case for Anosov flows (M. Handel and W. Thurston, \cite{Ha-Th}), but the known counter examples are very pathological, and this motivates the question, under which (mild) hypothesis are Anosov flows conjugated to an algebraic model? Some of the known results in this direction (E. Ghys  \cite{ghys0} and \cite{ghys1}, Y. Benoist, P. Foulon and F. Labourie \cite{BFL0} and \cite{BFL}, V.Sadovskaya \cite{Sadovskaya}, Y. Fang \cite{Fang2}and \cite{Fang1}) seems to imply that this should always be the case if the Anosov flow has some additional smoothness condition and/or is associated with some kind of geometric structure.

Motivated by this and, the still open, conjecture for Anosov actions of $\R^k$, $k\geq 2$, we propose a definition of a very general geometric structure and we show, following the same strategy as Benoist-Foulon-Labourie \cite{BFL}, that Anosov actions associated with this structure are in fact smoothly conjugated to affine actions (a slightly more general model then the algebraic ones).

To be precise, we show that  
\begin{theorem}\label{Theo1}
Every rigid $k$-geometric Anosov action is smoothly conjugated to an affine action.
\end{theorem}
As a corollary, we obtain:
\begin{theorem}\label{CoroTheo1}
Let $\phi:\R^k\times M \to M$ be an Anosov action, with smooth invariant bundles, which preserves a pseudo-Riemannian metric. Suppose that every Lyapunov distribution is one dimensional and the Lyapunov exponents $\chi_i$ satisfies the following non-resonance condition:
$$\chi_i + \chi_j+\chi_k = 0\Leftrightarrow \chi_i=0\;\;\chi_j = -\chi_k$$

Then, either $\phi$ is a suspension of an $\Z^k$ action (not necessarily Anosov) or, up to finite covers, $\phi$ is smoothly conjugated to an affine Anosov action.
\end{theorem}
We also show that a large family of examples of algebraic Anosov actions have an associated rigid $k$-geometric structure.
\begin{theorem}\label{Theo2}
Let $(G,H,\Gamma,\liea)$ be an algebraic Anosov action such that the sequence
$$0\to \mathfrak{Nil(g)}\to\lieg\to\lieg_{red}\to 0$$
splits. Then it is a rigid $k$-geometric Anosov action.
\end{theorem}

The case where $\mathfrak{Nil(g)}=0$, that is, $\lieg$ is reductive, was considered in a previous work \cite{Almeida1}, where it was shown that such actions are generalized $k$-contact Anosov actions, which is a particular case of rigid $k$-geometric Anosov action.

This paper is organized as follows:
\begin{itemize}
    \item On section \ref{introsec} we give some definitions and preliminary results. We motivate our most general geometric structure by first defining the stronger notion of "generalized $k$-contact structure", a natural generalization of the usual contact structures (\cite{Almeida1}).
    
    We also differentiate algebraic actions from affine actions. In particular, we can now understand the precise meaning of Theorem \ref{Theo1}, which is restated at the end of the section.
    
    \item On section \ref{examsec} we give some conditions which implies the existence of our geometric structure. In particular, we prove that under the hypothesis of Theorem \ref{CoroTheo1}, we have a ridig $k$-geometric structure. Together with Theorem \ref{Theo1}, this proves the Theorem \ref{CoroTheo1}. We also prove Theorem \ref{Theo2}. The main ingredients of this proof is a previous work (\cite{Almeida1}) where we consider the case where the sequence is trivial ($\mathfrak{Nil(g)} = 0$).
    
    Together with Maquera-Barbot's classification (\cite{Ba-Maq3}) of algebraic nilpotent Anosov actions, this shows that a large class of known Anosov actions of $\R^k$ are in fact rigid $k$-geometric.
    
    \item On sections \ref{gromsec} and \ref{Adap} we construct the main ingredients of the proof: A Gromov-rigid geometric structure, an adapted connection and some Lie algebras: $\lieg'$ and $\liek'$. 
    
    The Lie algebras are Lie algebras of Killing vector fields of the Gromov geometric structure and they will correspond to the Lie algebras of the Lie groups $\hat G'$ and $\hat H'$ which makes our homogeneous model $\hat G'\slash \hat H'$.
    
    \item Sections \ref{tecsec}, \ref{Nilsec} and \ref{moretecsec} are the most technical sections, where we carefully dissects the Lie algebras involved. The end goal is to show that $\hat H'$ is in fact closed in $\hat G'$ and therefore $\hat G'\slash \hat H'$ is actually a smooth manifold.
    
    \item On section \ref{modelsec} we find a suitable group $G''$ acting transitively on $\hat V:=\hat G'\slash \hat H'$ and construct a $(G'',\hat V)$ structure on an open dense subset of $M$, that is $\hat V$ is a local model for an open dense subset of $M$.
    
    \item On section \ref{extendsec} we extend the model from the open and dense subset to the whole manifold $M$. 
    
    \item Finally, on section \ref{completesec} we show that the $(G'',\hat V)-$structure, is in fact complete, that is, the developing map $\tilde M\to \hat V$ is in fact a covering map. This implies that $\tilde M$ is in fact an homogeneous manifold, and because $\hat V$ is simply connected, the developing map actually give us the desired conjugacy with an affine model.
\end{itemize}

\section{Main definitions}\label{introsec}

The goal of this section is to introduce the definitions involved in the statement of our main theorem. On subsection \ref{agaf} We recall the well known definitions of Anosov and algebraic actions and introduce the weaker (but quite natural) notion of affine action. On subsection \ref{kcont}, as a preparation to define our main geometric structure, we recall the notion of a generalized $k$-contact action which was introduced in \cite{Almeida1}. Finally, on subsection \ref{krigst} we introduce our main geometric structure and restate our main theorem, which we now have the tools to properly understand.

\subsection{Algebraic vs Affine}\label{agaf}
First let us recall the definition of an Anosov action:
\begin{definition}\label{anodef}
	Consider a compact smooth manifold $M$ and a smooth action $\phi:\R^k\times M \to M$. This action is said to be Anosov if there exists an element $a\in \R^k$, called an Anosov element, such that $\phi_a$ acts on $M$ normally hyperbolically, that is, there exists a, $d\phi^a$ invariant, continuous, splitting $TM = E^+\oplus T\phi\oplus E^-$, where $T\phi$ is the distribution tangent to the orbits, such that, there exists positive constants $C,\lambda$ for which
	\begin{align}\label{estimates}
		\|d\phi^{ta}(u^+)\|&\leq Ce^{-t\lambda}\|u^+\|\enskip\forall t>0\enskip\forall u^+\in E^+\\
		\|d\phi^{ta}(u^-)\|&\leq Ce^{t\lambda}\|u^i\|\enskip\forall t<0\enskip\forall u^-\in E^-
	\end{align}
\end{definition}
We recall that an \textit{algebraic action} is given by a quadruple $(G,H,\Gamma,\mathcal A)$ such that
\begin{itemize}
    \item $G$ is a Lie group with Lie algebra $\lieg$
    \item $H\subset G$ is a compact subgroup with Lie algebra $\lieh$
    \item $\Gamma\subset G$ is a uniform lattice acting freely on $G\slash H$, therefore $\Gamma\backslash G\slash H$ is a compact manifold
    \item $\mathcal A\subset \lieg$ is an abelian Lie subalgebra such that $\mathcal A \cap \lieh = \emptyset$ and $\mathcal A$ is contained in the normalizer $N_G(\lieh)$ of $\lieh$. Therefore, $\mathcal A$ acts on $\Gamma\backslash G\slash H$ by right multiplication.
\end{itemize}

An algebraic action is Anosov if, and only if, there exists an element $x\in \mathcal A$ and an $ad(x)$ invariant splitting $\lieg = \lieh\oplus\mathcal A\oplus \mathcal S\oplus\mathcal U$ such that the eigenvalues of $ad(x)$ on $\mathcal U$ (resp. $\mathcal S$) has positive (resp. negative) real part.

The most notable examples of Anosov algebraic actions are given by the Weyl chamber action: $(G,H,\Gamma,\mathcal A)$, where
\begin{itemize}
    \item $G$ is a connected semisimple Lie group with Lie algebra $\lieg$
    \item $\mathcal A\subset \lieg$ is a Cartan subspace, that is, an abelian Lie algebra of hyperbolic elements and maximal for those properties.
    \item It is known that the centralizer of $\mathcal A$ is of the form $\mathcal A\oplus \lieh$ where $\lieh$ is a Lie subalgebra which corresponds to a compact Lie subgroup $H\subset G$
    \item $\Gamma\subset G$ is a uniform lattice acting freely on $G\slash H$, which always exists (a classical result due to Borel).
\end{itemize}

For our purpose, we are forced to weaken slightly the notion of an algebraic action:
\begin{definition}
    An \textit{Affine action} is given by a quadruple $(G,H,\Gamma,\mathcal A)$ such that
\begin{itemize}
    \item $G$ is a Lie group with Lie algebra $\lieg$
    \item $H\subset G$ is a  \textit{closed} subgroup with Lie algebra $\lieh$
    \item $\Gamma\subset G$ is a \textit{discrete subgroup} acting freely on $G\slash H$, \textit{such that $\Gamma\backslash G\slash H$ is a compact manifold}.
    \item $\mathcal A\subset \lieg$ is an abelian Lie subalgebra such that $\mathcal A \cap \lieh = \emptyset$ and $\mathcal A$ is contained in the normalizer $N_G(\lieh)$ of $\lieh$. Therefore, $\mathcal A$ acts on $\Gamma\backslash G\slash H$ by right multiplication.
\end{itemize}
\end{definition}

\subsection{Generalized $k$-contact structures and contact Anosov actions}\label{kcont}

\begin{definition}
	A generalized $k$-contact structure on a manifold $M$ of dimension $2n+k$ is a collection of, $k$ non zero, linearly independent, $1$-forms $\{\alpha_1,\dots,\alpha_k\}$ and a splitting $TM = I\oplus F$, 	$dim I = k$,  such that, for every $1\leq j\leq k$ we have, 
	\begin{enumerate}
		\item {}$F = \bigcap_{i=1}^k\ker\alpha_i$
		\item {}$\ker(d\alpha_j) = I$
	\end{enumerate}
	We denote this structure by $(M,\alpha,TM = I\oplus F)$.
\end{definition}
\begin{remark}
	It follows from the definition that, for every $j = 1,\dots,k$ 	$(d\alpha_j)_{|_F}$ is non degenerate; in particular, this means that
	$$\alpha_1\wedge\cdots\wedge\alpha_k\wedge d\alpha_j^n\enskip\enskip\text{ is a volume form}.$$
\end{remark}
\begin{definition}
Let $(M,\alpha,TM = I\oplus F)$ be a generalized $k$-contact structure, and let $B = (b_{ij})\in GL(k,\R)$. Let $\beta_i = \sum_j b_{ij}\alpha_j$ and $\beta = (\beta_1,\dots,\beta_k) = B\alpha$.

If $(M,\beta,TM = I\oplus F)$ is a generalized $k$-contact structure, then it is called a reparametrization of $(M,\alpha,TM = I\oplus F)$.
 \end{definition}

\begin{lemma}\label{reebfields}
	For each $j$, there is a unique vector field $X_j\in \Gamma(M,I)$ such that $\alpha_i(X_j) = \delta_{ij}$. These vector fields are called Reeb vector fields.
	
	Moreover, the Reeb vector fields commute one with each other:
	$$[X_i,X_j]=0$$
\end{lemma}

\begin{definition}
For a given generalized $k$-contact structure the induced $\R^k$ action given by the Reeb vector fields will be called a contact action.
\end{definition}

\begin{remark}
	Let $\phi$ be a contact action. Notice that $I$ is precisely $T\phi$, the distribution tangent to the action. Moreover, if the 1-forms $\alpha_1,\dots,\alpha_k$ are smooth, then the splitting $TM = I\oplus F$ is smooth, so are the vector fields $X_j$ and therefore so is the action $\phi$.
	
	It is sometimes convenient to denote a generalized $k$-contact structure on $M$ as the $4$-tuple $(M,\alpha,\phi,F)$, where $\alpha = (\alpha_q,\dots,\alpha_k)$, $\phi$ denotes the contact action and $F\leq TM$ is the $\phi$-invariant sub-bundle where $d\alpha_j$ is non degenerate.
\end{remark}

\begin{definition}
	A generalized $k$-contact action $(M,\alpha,\phi,F)$ will be called a contact Anosov action, if some Reeb vector field $X_j$ defines an Anosov element of the induced contact action and the Anosov invariant distributions satisfies $ E^+\oplus E^- = E$. 
\end{definition}

\begin{remark}
    Notice that a contact Anosov action is topologically transitive for it preserves a volume form.
\end{remark}

\begin{remark}
It is well known that geodesic flows on negatively curved manifolds are contact Anosov flows. On the particular case of constant negative curvature, this geodesic flow is actually the Weyl chamber action $(G,H,\Gamma,\mathcal A)$ where $G$ is a semisimple Lie group of Rank 1. It is natural therefore to ask whether the Weyl chamber actions of higher rank is associated with a geometric structure of similar nature. The next theorem (\cite{Almeida1}) show us that the generalized $k$-contact structure is a natural candidate for this structure.
\end{remark}

\begin{theorem}
Let $(G,H,\Gamma,\mathcal A)$ be a Weyl chamber action. Then there exists a generalized $k$-contact structure on $\Gamma\backslash G\slash H$, such that the induced contact action is Anosov and it coincides with the Weyl chamber action.
\end{theorem}



\subsection{$k$-geometric Anosov actions}\label{krigst}
On this section, we define an even more general geometric structure and associated Anosov action.

\begin{definition}\label{moregeneral}
	An abelian action $\phi: \R^k\times M \to M$ will be called a $k$-geometric of type $(l,p,q)$, $p+q = n$, $l\leq k$ on a manifold $M$ of dimension $2n+k$ if there exists a basis for the action $\mathcal B = \{X_1,\dots,X_k\}$ and a sub-bundles $E_0\subset E^+\oplus E^-$, $dim(E_0)=p$, $T\phi_0\subset T\phi$ such that 
	\begin{enumerate}
	    \item {}$\mathcal B_0:=\{X_1,\dots,X_l\}$ span $T\phi^0$.
	    \item {}$T\phi^0\oplus E_0$ is integrable
	    \item {}The $1$-forms $\alpha_1,\dots\alpha_l$ dual to $\mathcal B_0$ have constant rank $2p+1$ 
		\item {}$(d\alpha_i)_{|_{E_0}}$, $i=1,\dots,l$ is non degenerate;
	\end{enumerate}
\end{definition}
\begin{lemma}
Let's denote by $T\phi_j\subset T\phi$ the sub-bundle spanned by $$\{X_1,\dots,X_{j-1},X_{j+1},\dots,X_k\}$$
For $j=1,\dots,l$ there exists a sub-bundle $E_j\subset E^+\oplus E^-$ complementary to $E_0$ and such that
The sub-bundle $T\phi^j\oplus E_j$ is integrable:
\end{lemma}
\begin{proof}
We define the sub-bundle $E_j:=\ker d\alpha_j\cap E^+\oplus E^-$.

Notice that the Anosov action preserves the splitting, and thus, if $Z$ is tangent to $E_j$, then $[X_s,Z]\in\Gamma(M,E_j)$, $1\leq s\leq k$. Thus, it remains to show that for vector fields $Y,Z$ tangent to $E_j$ the commutator $[Y,Z]$ is tangent to $T\phi^j\oplus E_j$, that is we just check that it belongs to the kernel of $d\alpha_j$ and $\alpha_j$. 
\begin{align*}
    i_{[Y,Z]}d\alpha_j &= [i_Y,\L_Z]d\alpha_j = i_Y(\L_Zd\alpha) - \L_Z(\underbrace{i_Yd\alpha_j}_{=0})\\
    &=i_Y(i_Z\circ \underbrace{d(d\alpha_j)}_{=0}+ d(\underbrace{i_Zd\alpha_j}_{=0}))=0
\end{align*}
also
\begin{align*}
    0 = d\alpha_j(Y,Z) = Y(\alpha_j(Z)) - Z(\alpha_j(Y)) - \alpha_j([Y,Z]) = - \alpha_j([Y,Z])
\end{align*}

as we wanted.

\end{proof}

\begin{remark}
	It can be shown that if the structure is of type $(k,p,0)$, then it is a generalized $k$-contact structure. In fact, the integrability of $T\phi^0\oplus E_0$ means that a rigid $k$-geometric structure of type $(l,p,q)$ is actually folliated by manifolds with generalized $l$-contact structure.
\end{remark}

\begin{definition}\label{genkcontAnosov}
	A $k$-geometric action of type $(l,p,q)$ $(M,\phi^0\subset\phi,TM = T\phi\oplus E_0\oplus E_j)$ will be called a $k$-geometric Anosov action if every element of $\mathcal B_0$ is Anosov, with the same invariant stable and unstable bundles, and they split as: 
	\begin{align}\begin{cases}
	E^\pm= E_0^\pm\oplus E_j^\pm\\
	E_0 = E_0^+\oplus E_0^-\\
	E_j = E_j^+\oplus E_j^-
	\end{cases}
	\end{align}
\end{definition}
\begin{remark}
Associated with a $k$-geometric Anosov action of type $(l,p,q)$ with splittings 
	$$TM = T\phi^0\oplus T\phi'\oplus E_0^+\oplus E_0^-\oplus E_j^+\oplus E_j^-$$
	we have associated projections
	\begin{align*}
	    p_j^\pm&:TM\to E_j^\pm\\
	    \Check p_j^\pm&:TM\to E_0^\pm
	\end{align*}
\end{remark}

\begin{definition}\label{rigidkgeo}
 A $k$-geometric Anosov action $(M,\phi^0\subset\phi,TM = T\phi\oplus E_0\oplus E_j)$  will be called rigid if there exists a $q$-form $\Theta$ and a linear, $\phi$-invariant, connection $\tilde \nabla^j$ on $T\phi^j\oplus E_j$ such that
 \begin{itemize}
     \item $\Theta$ restricts to a nowhere zero top form over $E_j$
     \item $\alpha_1\wedge\dots\wedge\alpha_k\wedge d\alpha_j^p\wedge\Theta$ is a volume form for every $j\in \{1,\dots,l\}$
     \item $\tilde\nabla^j E_j^\pm\subset E_j^\pm$
     \item $\tilde\nabla^j_ZY = p_j^\pm[Z,Y]\enskip\enskip\forall \;Z\in \Gamma(M,E_0)\;;\;Y\in \Gamma(M,E_j^\pm)$
     \item Every vector field $Y$ tangent to $T\phi^0\oplus E_0$ such that $[Y,X_s]=0$ for every $s$, is in fact a infinitesimal symmetry of $\tilde \nabla^j$.
     \item  $\tilde\nabla^j_{X_s}Z =[X_s,Z]$ for every $s=1,\dots,k$ and $Z$ tangent to $E_j$.
 \end{itemize} 
\end{definition}

We are now ready to state our main theorem.
\begin{theorem}\label{MainTheo}
Every rigid $k$-geometric Anosov action is smoothly conjugated to an affine action.
\end{theorem}

\section{Some examples}\label{examsec}
                \subsection{Pseudo-metric preserving.}
                    On this section we  present some conditions which implies that the action is $k$-geometric and rigid. Our main interest is an Anosov action which preserves a pseudo-Riemannian metric 

\begin{lemma}\label{2form}
An Anosov action with smooth invariant bundles, preserves a pseudo-Riemannian metric if, and only if, it preserves a 2-form $\omega$ such that $\ker(\omega) = Span_\R\{X_1,\dots,X_k\}$
\end{lemma}
\begin{proof}
We just consider an endomorphism $J:TM\to TM$ such that $J(T\phi)=0$, $J(u^\pm) = \pm u^\pm$ for $u\pm \in E^\pm$ and the induced 2-form $\omega(u,v) = g(Ju,v)$. On the other way, we define $g$ by
$$g(X_i,X_j) = \delta_{ij}\enskip;\enskip g(X_i,E^\pm)=0\enskip\enskip g(E^\pm,E^\pm)=0\enskip\enskip g(u^+,v^-) = \omega(u^+,v^-).$$
\end{proof}

Let $\alpha_1,\dots\alpha_k$ be 1-forms dual to a choice of basis $X_1,\dots,X_k$ for the Anosov action. In general, there is little to be said about the forms $d\alpha_j$, which is why we define
\begin{definition}
A choice of basis $X_1,\dots,X_k$ is called special if the dual forms $\alpha_1,\dots,\alpha_k$ satisfy:
\begin{enumerate}
    \item Each $\alpha_j$ has constant rank equal to $2p+1$.
\end{enumerate}
In general, an Anosov action may not admit an special basis, if it does, we shall say that the Anosov action has constant rank.
\end{definition}

Even for especial basis there is no relation between $d\alpha_j$ and $\omega$ for any $j$, that is why we define:
\begin{definition}
Let $\phi$ be an Anosov action of constant rank and $X_1,\dots,X_k$ a special basis.  Suppose that $\phi$ preserves a pseudo Riemannian metric $g$. The action is called adapted to $g$ if the dual forms $\alpha_1,\dots,\alpha_k$ satisfy:
$$\alpha_1\wedge\dots\wedge\alpha_k\wedge d\alpha_j^p\wedge\omega^{n-p}\enskip\text{ is a volume form}$$
If moreover, the restrictions of $\omega$ to $\ker(d\alpha_j)$ are closed for every $j$, we shall say that the action is rigidly adapted to $g$
\end{definition}

\begin{lemma}
An Anosov action rigidly adapted to $g$ is a rigid $k$-geometric Anosov action of type $(k,p,q)$
\end{lemma}
\begin{proof}
Take $\omega_j$ the restriction of $\omega$ to $\ker(d\alpha_j)$. As $\alpha_1\wedge\dots\wedge\alpha_k\wedge d\alpha_j^p\wedge\omega^{n-p}$ is a volume form, we obtain that $\omega_j$ is non degenerate over $\ker(d\alpha_j)\cap\bigcap\ker(\alpha_s)$ and thus $\ker (\omega_j)=T\phi\oplus E_0  $ where $E_0$ is some sub-bundle of $E^+\oplus E^-$ which is complementary to $E_j$ (Notice that $E_0$ also depends on $j$ and thus, it would be more correct to write $E_{0,j}$, but we omit the extra index).

As the action is rigidly adapted, $d\omega_j = 0$ and thus, some easy computations shows that $\ker(\omega_j)$ is actually integrable.

We also remark that the hyperbolic dynamics implies that $E_0 = E_0^+\oplus E_0^-$, where $E_0^\pm = E_0\cap E^\pm$ and the same for $E_j:=\ker(d\alpha_j)\cap\bigcap_s\ker\alpha_s$.

Now, to construct the connection on $E_j$ we modify slightly the Kanai connection \cite{Kanai} and we observe that there exists a unique connections $\nabla^\pm$ on $T\phi^j\oplus E_j^+\oplus E_j^-$ such that
\begin{itemize}
    \item $\nabla\omega_j = 0$ and $\nabla E_j^\pm \subset E_j^\pm$
    \item $\nabla X_s = 0$
    \item $\nabla_{Z_0}Z_j = p_j[Z_0,Z_j]$ for any $Z_i$ tangent to $E_i$.
    \item $\nabla_{X_s}Z_j = [X_s,Z_j]$
    \item $\nabla_{Y_j^\mp}Z_j^\pm = p_j^\pm[Y_j^\mp,Z_j^\pm]$
\end{itemize} 

This connection is, of course, invariant by the action.\\

Also, we define the volume form $\Theta = \omega_j^{n-p}$ which is of course invariant by $\phi$, parallel with respect to $\nabla$ and restricts to a volume form over $E_j$.

It remains to verify that vector fields which commute with $X_s$ and are tangent to $T\phi_0\oplus E_0$ are in fact infinitesimal symmetries for this connection. For this, we observe that if a vector field $Y$ commute with $X_s$, it preserves the splitting $TM = T\phi\oplus E_0^+\oplus E_0^-\oplus E_j^+\oplus E_j^-$, moreover, if $Y$ is tangent to $T\phi\oplus E_0$ then $i_Y\tilde\omega=0$ and
\begin{align*}
    \L_Y\tilde\omega = d\circ i_Y\tilde\omega + i_Y\circ d\tilde\omega = 0
\end{align*}

It is clear, henceforth, that $Y$ is an infinitesimal symmetry for $\nabla$
\end{proof}

Now, using the same arguments given by Y. Fang in \cite{Fang2}, we can prove the following lemma:
\begin{lemma}\label{1dlyap}
Let $\phi$ be an Anosov action with smooth invarian bundles $E^\pm$ which preserves a smooth pseudo-riemannian metric $g$. Suppose $E^+\oplus E^-$ is non integrable. Suppose, moreover, that every Lyapunov distribution of $\Phi$ is one dimensional and the Lyapunov exponents $\chi_j$ satisfy the following non resonance condition:
\begin{align}\label{nrc}
    \chi_i + \chi_j+\chi_k = 0\Leftrightarrow \chi_i=0\;\;\chi_j = -\chi_k
\end{align}
then $\phi$ is a rigid $k$-geometric Anosov action. 
\end{lemma}
Before we proceed with the proof, we state the following Lemma Adapted from Feres-Katok
\begin{lemma}\label{fereska}
Let $\phi$ be an ergodic action on a compact manifold $M$, and let $$T_xM  = E_{\chi_1}(x)\oplus \dots\oplus E_{\chi_l}(x)$$
be its Oseledets decomposition (defined on a set of full measure $\Lambda$).
Let $x\in \Lambda$, and $\tau$ a continuous $\phi$-invariant tensor field of type $(0,r)$. Let $t_j \in E_{\chi_{i_j}}(x)$, $j = 1,\dots,r$, such that $\tau_x(t_1,\dots,t_r) \neq 0$, then $\sum_{j=1}^r\chi_{i_j} = 0$
\end{lemma}

\begin{proof}[Proof of Lemma\ref{1dlyap}]
From the previous discussion, it suffices to show that the action is rigidly adapted to $g$. Our first step is to construct a special basis.\\

Let $\alpha$ be a 1-form on $M$ dual to the action such that $d\alpha_x$ has the maximal possible rank for some point $x$. The non-integrability of $E^+\oplus E^-$ implies that this maximal rank $s$ is non zero, that is, $d\alpha^s$ is non zero for some choice of $\alpha$. We define
$$U:=\{x\in M\;;\; d\alpha^s_x\neq 0\}$$
As the action is topologically transitive, $U$ is open and dense. We define $$F_x:=\{Y\in T_xM\;;\;d\alpha_x(Y,\;\cdot\;) = 0\}$$
and $F = \cup_{x\in U}F_x$

As the metric $g$ is $\phi$-invariant tensor field of type $(0,2)$. The non resonance conditions implies that Lyapunov exponents comes in pairs, one positive and one negative and thus, we write the Oseledets decomposition as

$$T_xM  = L_1^+\oplus L_1^-\oplus\dots\oplus L_r^+\oplus L_r^-.$$
where, for any vectors $l_a,l_b$ tangent to some Lyapunov distribution, we have $g(l_a,l_b)=0$ except for the case $l_a$ is tangent to $L_i^+$ and $l_b$ is tangent to $L_i^-$ for some $i$. Of course, the same is true for the 2 form $\omega$.

As $d\alpha$ is also a  $\phi$-invariant tensor field of type $(0,2)$. We obtain, after renaming the Lyapunov exponents, for $x\in U\cap\Lambda$ and $l_j^\pm$ tangent to $L_j^\pm$,
\begin{align*}
    d\alpha_x(l_i^\pm,l_j^\pm) &= 0\enskip\enskip\forall i,j\\
    d\alpha_x(l_i^\pm,l_j^\mp) &= 0\enskip\enskip\forall i,\neq j\\
    d\alpha_x(l_i^\pm,l_i^\mp)& = 0\enskip\enskip\forall i >\frac{Rank(d\alpha_x)}{2}\\
    d\alpha_x(l_i^\pm,l_i^\mp)& \neq 0\enskip\enskip\forall i \leq\frac{Rank(d\alpha_x)}{2}
\end{align*}

Taking $\alpha = \alpha_1,\dot,\alpha_k$ a basis of $T^*\phi$, it is clear that 
$$\alpha_1\wedge\dots\wedge\alpha_k\wedge d\alpha^s\wedge\omega^{n-s}$$
is a non zero $\phi$-invariant top form. Because $\phi$ is topologically transitive, we obtain that $\alpha_1\wedge\dots\wedge\alpha_k\wedge d\alpha^s\wedge\omega^{n-s}$ is a volume form. This means that $d\alpha^s$ vanishes nowhere and thus, $U=M$. In particular, the 1-form $\alpha$ has constant rank $2s+1$.

Now, we observe that non degeneracy is an open condition and therefore, if we choose 1-forms close to $\alpha$ they will also be of the same (maximal possible) rank. We can thus construct a special basis $X_1,\dots,X_k$. Using the same arguments as before, we conclude that 
$$\alpha_1\wedge\dots\wedge\alpha_k\wedge d\alpha_j^s\wedge\omega^{n-s}$$
is a volume form for every $j$;

Finally, we observe that any, $\phi$-invariant, $3$-form $\eta$ on $M$ such that $i_{X_s}\eta=0$ is identically zero. This follows from the non resonance condition (\ref{nrc}) and Feres-Katok theorem.

\end{proof}

\begin{remark}
We remark that the case where $E^+\oplus E^-$ is integrable was briefly considered by Barbot-Maquera \cite{Ba-Maq2}, where they proved that such actions are suspensions of $\Z^k$ actions on a compact manifold, though not necessarily an Anosov $\Z^k$ action.
\end{remark}
                \subsection{Algebraic examples}
                    On this section, we show that a large class of algebraic Anosov actions are in fact rigid $k$-geometric Anosov actions.
Consider an  algebraic Anosov action $(G',H',\Gamma',\liea')$. Suppose that the Levi factor\footnote{The Levi factor is the semisimple part of the Levi decomposition} of $G'$ is non compact. Let us denote by $\liek' = \mathfrak{Nil(g')}$ the nilradical of $\lieg'$. We have the corresponding exact sequence:
\begin{align}\label{exactalgebra}
    0\to\liek'\to\lieg'\to\lieg\to 0
\end{align}
where $\lieg$ is reductive. We take the induced exact sequence of Lie groups:
$$1\to K'\to G'\to G \to 1$$ where $\lieg = Lie(G)$.

Following Barbot-Maquera (\cite{}), there exists an associated algebraic Anosov action $(G,H,\Gamma,\liea)$ where $H'\subset G'$ is a compact lift, with corresponding Lie algebra $\lieh'\subset\lieg'$, such that $\lieh'\cap\liek' = \{0\}$. Moreover, $\liea'\subset\lieg'$ is also an abelian lift of $\liea$ and $\Gamma'\subset G'$ is an uniform lattice that projects to $\Gamma$. In the language of \cite{}, $(G',H',\Gamma',\liea')$ is a nil suspension of $(G,H,\Gamma,\liea)$.

\begin{remark}
Notice that the exact sequence induces a principal $K'$-bundle $G'\slash H'\to G\slash H$ and a $\Lambda\backslash K'$-bundle $\Gamma'\backslash G'\slash H'\to \Gamma\backslash G\slash H$, where $\Lambda = \Gamma'\cap K'$.
\end{remark}

Because $G$ is reductive, it is commensurable to a central extension of a modified Weyl chamber action. This last action admits a compatible algebraic generalized $l$-contact structure, that is, the 1-forms which defines the contact structure are in fact induced from left invariant 1-form on $G$. Because commensurability is obtained by algebraic steps, it is easy to see that the algebraic Anosov action  $(G,H,\Gamma,\liea)$ is in fact an algebraic generalized $l$-contact Anosov action.\\

We want to construct a rigid $k$-geometric structure on $(G',H', \Gamma', \liea')$. That may not always be possible, but we will find some sufficient conditions for this to be true.

\begin{theorem}
Suppose that the exact sequence \ref{exactalgebra} splits, then $(G',H',\Gamma',\liea')$ has rigid $k$-geometric compatible structure.
\end{theorem}
The proof will be broken in several Lemmas, first of all we notice that because the sequence splits, then the lift $H'$ of $H$ must have Lie algebra $\lieh$, in fact, the lie algebra $\lieh'$ must contain $\lieh$ for it is a lift, and the condition $\lieh'\cap\liek' = \{0\}$ implies that $\lieh'$ has no other direction but those of $\lieh$, that is, $\lieh' = \lieh$.
\begin{lemma}
Consider the affine Anosov action given by $(\hat G',\hat H',\hat \Gamma', \liea')$ where $\hat G'$ is the universal cover of $G'$ and $\hat H'$ and $\hat\Gamma'$ the corresponding subgroups. Then it is a $k$ geometric action.
\end{lemma}
\begin{proof}
     First we observe that, since the sequence   \ref{exactalgebra} splits, we can identify $\hat G'$ with $\tilde K'\ltimes \tilde G$ where $\tilde K'$ and $\tilde G$ are the simple connected connected subgroups associated with $\liek'$ and $\lieg$, in particular, $\tilde G$ and $\tilde K'$ are the universal cover of $G$ and $K'$.

    We also observe that as $\liea'$ is a lift of $\liea$, then, as vector spaces, $\liea' = \liea\oplus\liea^1$, where $\liea^1\subset\liek'$.
    
    Finally, we observe that, because the affine action $(\hat G',\hat H',\hat \Gamma', \liea')$ is Anosov, we have a splitting
    $$\lieg':=\lieh'\oplus\liea'\oplus \mathcal S'\oplus\mathcal U'$$
    which is $ad(v_0)$-invariant for some element $v_0\in\liea'$ and, moreover, on $\mathcal S'$ and $\mathcal U'$ has eigenvalues of positive and negative real parts respectively. Because $\lieg' = \liek'\ltimes\lieg$ it follows that this splitting can be further refined as
      $$\lieg':=\lieh\oplus\liea\oplus\liea^1\oplus \mathcal S_0\oplus\mathcal S_1\oplus\mathcal U_0\oplus\mathcal U_1$$
    where $\liek':=\liea^1\oplus \mathcal S_1\oplus\mathcal U_1$ and $\lieg:=\lieh\oplus\liea\oplus \mathcal S_0\oplus\mathcal U_0$. Which is also $ad(v_0)$-invariant.
    
    Such splitting, of course, descends to a splitting
    $$T\big(\hat\Gamma\backslash \hat G \slash \hat H\big):=T\phi^0\oplus T\phi^1\oplus \mathcal E_0^+\oplus\mathcal E_1^+\oplus\mathcal E_0^-\oplus\mathcal E_1^-$$
    
    The integrability of $T\phi^0\oplus \mathcal E^+_0\oplus\mathcal E^-_0$ follows from the fact that $\lieg$ is a Lie subalgebra of $\lieg'$.
    
    Now, let $\alpha_1,\dots,\alpha_l$ be the left invariant 1-forms on $G$ (and also on $\tilde G$) that induces the generalized $l$-contact action on $(G,H,\Gamma,\liea)$ , and let $\pi:\hat G' = \tilde K'\ltimes \tilde G\to \tilde G$ the projection on the second factor, and $\tilde \alpha_j := \pi^*\alpha_j$.
    
    It is clear that $d\tilde\alpha_j$ is non degenerate over $\mathcal S_0\oplus\mathcal U_0$ and $\ker(d\alpha_j)=\lieh\liea'\oplus \mathcal S_1\oplus\mathcal U_1$, and thus, such forms descends to left invariant forms $\tilde \alpha$ over $\hat\Gamma\backslash \hat G \slash \hat H$ such that
    \begin{itemize}
        \item $d\tilde\alpha_j$ is non degenerate over $\mathcal E_0^+\oplus\mathcal E_0^-$
        \item $\ker(d\alpha_j) = T\phi\oplus \mathcal E_1^+\oplus\mathcal E_1^-$
    \end{itemize}
    as we desired.

\end{proof}
The following Lemma is classical
\begin{lemma}
There exists a bi-invariant connection on $\hat G'$ defined by
$$\nabla'_YZ = [Y,Z]\enskip\enskip \forall\;Y,Z\;\;left\;invariant$$
\end{lemma}
Because this connection is bi-invariant, it descends to a left invariant connection $\nabla''$ on $\hat G'\slash \hat H'$.
\begin{lemma}
Suppose that the Anosov action $(G',H',\Gamma',\liea')$ is topologically transitive, then there exists a left invariant, $\liea'$ invariant, connection $\tilde\nabla$ on $\hat G'\slash \hat H'$ that preserves the splitting 
\begin{align}\label{splitting}
    T\big( \hat G \slash \hat H\big):=T\phi^0\oplus T\phi^1\oplus \mathcal E_0^+\oplus\mathcal E_1^+\oplus\mathcal E_0^-\oplus\mathcal E_1^-
\end{align}
\end{lemma}

\begin{proof}
We just define, for any vector field $Y_{i}^\pm$ tangent to $\mathcal E_i^\pm$ and a vector field  $Z$:

$$\tilde\nabla_ZY_i^\pm = p_i^\pm \nabla'_ZY_{i}^\pm$$
and for the vector fields $X_s$
$$\tilde\nabla X_s = 0$$.

It is clear that this defines a new connection on $ \hat G \slash \hat H$. Because $\liea'$ is in the normalizer of $\hat H$, and the connection $\nabla'$ is bi-invariant,it follows that the induced connection $\nabla'$ on $\hat G\slash\hat H$ is $\liea'$ invariant. Moreover, as $\liea'$ preserves the splitting, $\tilde \nabla$ is $\liea'$ invariant.\\

It remains to show that it is also left invariant, that is, for any $g\in \hat G$ we must show
$$(L_g)_*(\tilde\nabla_ZY) = \tilde\nabla_{(L_g)_*Z}(L_g)_*Y$$

This is a straightforward computation using of the following facts:
\begin{itemize}
    \item If $\Gamma(\hat G, \mathcal S_i)$ and $\Gamma(\hat G, \mathcal S_i)^L$ denotes the space of vector fields tangent to $\mathcal S^i$ and the subspace of left invariant vector fields tangent to $\mathcal S^i$. Then $\Gamma(\hat G, \mathcal S_i)$ is precisely the $C^\infty(\hat G)$ module generated by $\Gamma(\hat G, \mathcal S_i)^L$.
    \item For any left invariant vector field $Y$ on $\hat G$ the projections $\pi_{\mathcal S_i}$ and $\pi_{\mathcal U_i}$ are also left invariant vector fields.
    \item For any vector field $Z$ on $\hat G\slash \hat H$, and $g\in \hat G$, if $\hat Z$ is a lift of $Z$, then $(L_g)_*\hat Z$ is a lift of $(L_g)_*Z$
\end{itemize}
\end{proof}
\begin{remark}
For vector fields $Y,Z$ tangent to $\mathcal E_0$ and $\mathcal E_1^+$ respectively, let $\hat Y$ and $\hat Z$ denote lifts which are tangent to $\mathcal S_0\oplus\mathcal U_0$ and $\mathcal S_1$ respectively. If we write $\hat Z = \sum f_i\hat Z_i$ and $\hat Y = \sum g_j\hat Y_j$ where $Z_i$ and $Y_j$ are left invariant and also tangent to $\mathcal S_0\oplus\mathcal U_0$ and $\mathcal S_1$. Then we have the following
\begin{align*}
    \tilde\nabla_ZY &=p_1^+\big(\nabla'_{\hat Z}\hat Y\big)\\
    &=p_1^+\big(\sum_{ij} f_iZ_i(g_j)Y_j + f_ig_j[Z_i,Y_j]\big)\\
    &p_1^+\big(\sum_{ij} f_iZ_i(g_j)Y_j + f_ig_j[Z_i,Y_j]\big) - \underbrace{p_1^+(g_jY_j(f_i)Z_i)}_{=0}\\
    &=p_1^{+}([Z,Y])
\end{align*}
\end{remark}

\begin{lemma}\label{existence}
There exists a left invariant $q$-form on $G'\slash H'$, $q = dim(\liek')$,  which restricts to a volume form on the fibers of
$$G'\slash H'\to G\slash H$$
\end{lemma}
Before we prove this Lemma, we recall a classical result
\begin{lemma}
Consider $G'$ a connected Lie group and $H'$ a closed subgroup. If $H'$ is compact, then the homogeneous space $G'\slash H'$ is reductive, that is, there exists a subspace $\mathfrak m\subset\lieg'$ such that $\lieg' = \lieh'\oplus\mathfrak m$ and, for every $h\in H'$ we have:
$$Ad_h(\mathfrak m)\subset\mathfrak m$$
\end{lemma}

\begin{proof}[Proof of Lemma \ref{existence}]
First we remark the equivalence
$$\Omega^q(G'\slash H')^G \equiv \Lambda^q(Hom(\lieg'\slash\lieh',\R)^{H'})$$
between left invariant forms on $G'\slash H'$ and it's linear counterpart, the $H'$-invariant alternating q-linear forms on $\lieg'\slash \lieh'$. \\

Let $o$ denote the coset $H'$ in $G'\slash H'$ and consider the isotropy representation
\begin{align*}
\chi:=\chi^{G'\slash H'}:H'&\to Aut(T_o(G'\slash H'))\\
                h&\mapsto (dL_h)_o    
\end{align*}
where $L_h$ denotes the left multiplication by $h$. From the splitting 
$$\lieg':=\lieh\oplus\liea\oplus\liea^1\oplus \mathcal S_0\oplus\mathcal S_1\oplus\mathcal U_0\oplus\mathcal U_1$$
we identify (as vector spaces)
$$T_o(G'\slash H') = \mathfrak m = \liea\oplus\liea^1\oplus \mathcal S_0\oplus\mathcal S_1\oplus\mathcal U_0\oplus\mathcal U_1$$

With this identification, we have 
$$\chi(h) = (dL_h)_o   = Ad_h^{H'}$$
where $Ad_h^{H'}$ is the induced adjoint on $\lieg'\slash\lieh' - \mathfrak m$.

We consider $\tilde\Theta \in \Lambda^q(Hom(\mathcal S_1\oplus\mathcal U_1,\R))$ where $q = dim (\mathcal S_1\oplus\mathcal U_1)$, and define, for tangent vectors $v_1,\dots,v_q\in\mathcal S_1\oplus\mathcal U_1$
\begin{align*}
    \Theta(v_1,\dots,v_q):= \int_{\chi(H')}\tilde\theta(\chi(h)v_1,\dots,\chi(h)v_q)d\chi(H')
\end{align*}
where $d\chi(H')$ is a left invariant volume form on $\chi(H')$. We remark that as $H'$ is compact, then so is $\chi(H')$ and thus the integral above is well defined.

Clearly, $\Theta$ is also belongs to $\Lambda^q(Hom(\mathcal S_1\oplus\mathcal U_1,\R))$. It remains to show that it is $H'$ invariant. We have:
\begin{align*}
    (h'\cdot \Theta)(v_1,\dots,v_q)&:=\Theta(Ad_{h'}(v_1),\dots,Ad_{h'}(v_q)\\
    &=\int_{\chi(H')}\tilde\theta(Ad_{h'}\chi(h)v_1,\dots,Ad_{h'}\chi(h)v_q)d\chi(H')\\
    &=\int_{\chi(H')}\tilde\theta(Ad_{h'}Ad_hv_1,\dots,Ad_{h'}Ad_hv_q)d\chi(H')\\
    &=\int_{\chi(H')}\tilde\theta(\chi(h'h)v_1,\dots,\chi(h'h)v_q)d\chi(H')
\end{align*}
As the volume form $d\chi(H')$ is left invariant the result follows.
\end{proof}

 \section{Gromov's geometric structure}\label{gromsec}
            We shall now, define a rigid geometric structure (in the sense of Gromov) associated with our geometric Anosov action. We recall that a Gromov A-structure $\sigma$ of order $r$ and type $(\lambda,\Sigma)$ on $M$ is given by an algebraic manifold $\Sigma$ with an algebraic action $\lambda:Gl^r(\R^m)\times\Sigma\to\Sigma$ and a $Gl^r(\R^m)$ equivariant map $\sigma: F^r(M)\to\Sigma$.

Let $m=2(p+q)+k$ and $V = \R^m$. We want to define an A-structure of order $2$ on $M$, but first let us understand a little better the group $Gl^2(V)$ and it's action on $F^2(M)$.

First, we will understand the elements of $F^2(M)$ as pairs $(\xi_p,B_p)$ where
\begin{align*}
    \xi_p:V\to T_pM\enskip&\enskip \text{isomorphism}\\
    B_p:V\times V\to T_pM\enskip&\text{bilinear}
\end{align*}

In a similar way, we understand the group $ Gl^2(V) := \{j^2_f(0)\;\;f:V\to V f(0) = 0 Df(0)\;\text{invertible}\}$ as pairs $(u,A)$ where $u:V\to V$ is an isomorphism and $A:V\times V\to V$ is bilinear. With this notation, the group law is given by
$$(u,A)\cdot(u',A') = (u\circ u',A\circ(u'\times u') + u\circ A')$$
and it right action on $F^2(M)$ is given by
$$(\xi_p,B_p)\cdot(u,A) = (\xi_p\circ u,B_p\circ(u\times u) + \xi_p\circ A)$$

Now, to define $\Sigma$, denote by $Gr^n(V)$ the Grassmanian of $n$-planes in $V$ and define:

\begin{align*}
\Sigma:=\bigg\{(X_1,\dots,X_k,e_0^+,e_0^-, e_1^+,e_1^-,\dots,e_l^+,e_l^-,\omega_1,\dots,\omega_l,B_1,\dots,B_l),\;\text{with}\\
X_j&\in V\;\forall\;j\;\enskip;\enskip e_0^\pm\in Gr^p(V)\;\enskip;\enskip e_j^\pm\in Gr^q(V)\;\text{such that}\\
&\enskip V = \R X_1\oplus\dots\oplus \R X_k\oplus e_0^+\oplus e_0^-\oplus e_j^+\oplus e_j^-\\
\omega_j& \in \Lambda^2(V^*)\;\text{such that, for each }\; j\in\{1,\dots,l\}\\
&\enskip\ker\omega_j = \R X_1\oplus\dots\oplus\R X_k\oplus e_j^+\oplus e_j^-\\
B&:V\times \mathcal E_j \to \mathcal E_j \;\text{bilinear map such that}\\
&\enskip\mathcal E_j = \R X_{1}\oplus \dots\oplus \R X_{j-1}\oplus\R X_{j+1}\oplus \dots\oplus \R X_k\oplus e_j^+\oplus e_j^-
\bigg\}
\end{align*}

The inspiration for this structure is more or less clear, except, perhaps, for the bilinear maps $B_j$. We recall that we are interested in rigid $k$-geometric Anosov actions and, therefore, the action comes with invariant connections $\tilde\nabla^j$ on the $T\phi^j\oplus E_j$, this connection have associated Christoffel symbols $\Gamma_{st}^k$ which will be understood as the bilinear map $B$.

Now, to construct the right action of $Gl^2(V)$ on $\Sigma$. 

Let $(X_1,\dots,X_k,e_0^+,e_0^-, e_1^+,e_1^-,\dots,e_l^+,e_l^-,\omega_1,\dots,\omega_l,B_1,\dots,B_l)\in\Sigma$ and $(u,A)\in Gl^2(V)$. The linear part $u$ acts on $X_j$, $e_i^\pm$ and $\omega_j$ in the obvious way, by changing coordinates. It remains to define $B_j\cdot (u,A)$. To simplify, we omit the index $j$ and write $B = B_j$. We are tempted to define it as
$$B\cdot (u,A) := u^{-1}\circ B\circ(u\times u) + u^{-1}\circ A$$
which is the usual formula for the change of coordinates for the Christoffel Symbols. However, the bilinear map must be $(e_j^+\oplus e_j^-)$ valued and defined on $V\times (e_j^+\oplus e_j^-)$. We must correct the domains and counter domains. Let $\pi:V\to (e_j^+\oplus e_j^-)$ be the projection (with respect to the splitting $V = \R X_1\oplus\dots\oplus \R X_k\oplus e_0^+\oplus e_0^-\oplus e_j^+\oplus e_j^-$) and $u_j$ the restriction of $u$ to $e_j^+\oplus e_j^-$. Then we define
$$B\cdot (u,A) := (u_j^{-1}\circ B\circ(u\times u_j) + u_j^{-1}\circ\pi\circ A$$

It is easy to check that we have in fact defined a $Gl^2(V)$ right action. It remains to construct the $Gl^2(V)$-invariant map $\sigma:F^2(M)\to\Sigma$. Before we do this, we consider the following procedure. Consider a splitting $\R^m = E\oplus F$, we want to construct a basis for $E$. Let $\pi\R^m\to E$ be the projection, and $\{e_1,\dots,e_m\}$ the canonical basis for $\R^m$. Let $f_i = \pi(e_i)$, then the set $\{f_1,\dots,f_m\}$ of course generates $E$. Let $i_0 \in \{1,\dots,m\}$ be the smallest number such that $$\{f_1,\dots,f_{i_0}\}$$ is linearly dependent. Now, take $i_1\in\{1,\dots,i_0-1,i_0+1,\dots,m\}$ such that 
$$\{f_1,\dots,f_{i_0-1},f_{i_0+1},\dots f_{i_1})\}$$
is linearly dependent. As $E$ has finite dimension, this process eventually finishes, and we obtain a basis for $E$, which allow us to identify $E = \R^l$ for some $l$. Doing the same with $F$, there exists a "canonical" way to identify $E\oplus F$ with $\R^l\oplus \R^{m-l}$. In particular, for a splitting $TM = \bigoplus_i E_i$ we can always understand a linear map $\xi_p:\R^m\to T_p^m$ as the direct sum of isomorphisms $\R^{d_i}\to E_i$.

Now, to define the map $\sigma$ we just define $\sigma(\xi_p,B_p)$ as the pullback (via $\xi_p$) of the tensors and the bilinear map as the Christoffel Symbols of the connection $\tilde\nabla^j$, where we use the previous procedure to obtain a local frame of $TM$ and $T\phi^j\oplus E_j^+\oplus E_j^-$.

        \section{An adapted connection}\label{Adap}
            On this section, we develop some technical preparation for studying what will turn out to be the Lie algebra of $G'$ on our intended $(G',G'\slash H')$-structure on $M$. This section follows closely the work of Benoist-Foulon-Labourie, and most of the proofs in this section are (somewhat more technical and cumbersome) adaptation of the original proofs on \cite{BFL} and were included for the sake of completeness and clarity.

Consider $k$-geometric Anosov action on $M$,
\begin{lemma}\label{buildconnec}
	For every $j \in \{1,\dots,l\}$, and real valued linear functionals $C_{j}^{\pm}:\R^k\to\R$  there exists unique smooth connection $\nabla^j$ on $T\phi^0\oplus E_0^+\oplus E_0^-$ that satisfies
	\begin{align}
	\nabla^j d\alpha_j &= 0\enskip;\enskip 
	\nabla^j \alpha_i=0\enskip\forall i\enskip;\enskip
	\nabla^j(E_0^{\pm})\subset E_0^{\pm}   \enskip;\enskip\nabla^j T\phi^0\subset T\phi^0 
	\end{align}
and for $X\in \Gamma(M,T\phi)$, $Z_i^\pm\in \Gamma(M,E_i^\pm)$, $i\in\{0,1\}$
	\begin{align}
	\nabla^j_{Z_1}Z_0^{\pm} &=\Check p_j^{\pm}([Z_1,Z_0^{\pm}]) \\ 
	\nabla^j_{Z_0^{\mp}}Z_0^{\pm} &=\Check p_j^{\pm}([Z_0^{\mp},Z_0^{\pm}]) \\
	\nabla^j_{X}Z_0^{\pm} &=[X,Z_0^{\pm}] + C_{j}^{\pm}(X)Z_0^{\pm} 
	\end{align}
	Where $\Check p_j^\pm$ denotes the projection $TM\to E_0^\pm$ with respect to the splitting
	$$TM = T\phi\oplus E_0^+\oplus E_0^-\oplus E_j^+\oplus E_j^-$$ 
	and we use the natural identification $\R^k = T_p\phi = \R X_1(p)\oplus \dots \oplus \R X_k(p)$ to compute $C^{\pm}(X)$. Moreover, because $\alpha_s$ vanishes outside of $T\phi\oplus E_0^+\oplus E_0^-$, we can consider then as elements of $\Lambda^1 (T\phi\oplus E_0^+\oplus E_0^-)^*$. Similarly for $d\alpha_j$.
\end{lemma}

\begin{proof}
	First, we shall prove that a connection $\nabla^j$ wich satisfies the hypothesis of the Lemma \ref{buildconnec}, also satisfies $\nabla^j X_s = 0$ for every $1\leq sl\leq l$. For the sake of clearer notation, we shall avoid using the $j$ superscript on the connection, and simply denote $\nabla$ instead of $\nabla^j$.
	
	 From $\nabla \alpha_i =0$ we have, for any vector fields $Y, Z$,
	$$0 =(\nabla_Y\alpha_i)(X_s) = Y(\alpha_i(X_s)) - \alpha_i(\nabla_YX_s) = - \alpha_i(\nabla_YX_s)$$
	
	As $\nabla T\phi^0\subset T\phi^0$ it follows that $\nabla_YX_s=0$.
	
	Now, we just have to define $\nabla_{G_i^{\pm}}Z_0^{\pm}$.  For this, we notice that $\nabla d\alpha_j=0$ means that for every vector field $A,B,C$ we have
	$$0 = (\nabla_Cd\alpha_j)(A,B) = \L_C(d\alpha_j(A,B)) + d\alpha_j(\nabla_CA,B) - d\alpha_j(A,\nabla_CB)$$
	But ${d\alpha_j}$ restricted to $E^+\oplus E^-$ is non degenerate, thus we define $\nabla_{G_i^{\pm}}Z_0^{\pm}$ to be the unique vector field tangent to $E_0^\pm$ which satisfies, for every $Y_0^\mp\in \Gamma(M,E_0^\mp)$.
\begin{align*}
	{d\alpha_j}(\nabla_{G_i^{\pm}}Z_0^{\pm}, Y_0^\mp) 
	= 
	\L_{G_i^{\pm}}({d\alpha_j}(Z_0^{\pm}, Y_0^\mp))  - 	
	{d\alpha_j}(Z_0^{\pm},\nabla_{G_i^{\pm}} Y_0^\mp) 
\end{align*}
\end{proof}
\begin{remark}\label{remcon}
	Notice that the smoothness of the splitting means that this connection is smooth.
\end{remark}
\begin{remark}
    If we consider a rigid $k$-geometric Anosov action, with connections $\tilde\nabla^j$ on $E_j^+\oplus E_j^-$, then we can use the above connections $\nabla^j$ to define new connections $\nabla^{j\prime}$ on $TM$ in the usual way:
    $$\nabla^{j\prime}_A(B+C) = \nabla^j_AB + \tilde\nabla^j_AC$$ $$A\in \Gamma(M,TM)\;;\; B\in\Gamma (M,T\phi^0\oplus E_0^+\oplus E_0^-)\;;\; C\in\Gamma(M,T\phi^j\oplus E_1^+\oplus E_1^-)$$
    We observe that on the intersection $T\phi^0\cap T\phi^j$ the connections coincide, and the formula above is well defined. We will write simply $\nabla^j$ for this new connections.
    This new connections will satisfy:
    \begin{align}
	\nabla^j d\alpha_j &= 0\enskip;\enskip 
	\nabla^j \alpha_i=0\enskip\forall i\enskip;\enskip
	\nabla^j(E_i^{\pm})\subset E_i^{\pm}   \enskip;\enskip\nabla^j T\phi^j\subset T\phi^j 
	\end{align}
and for $X\in \Gamma(M,T\phi)$, $Z_i^\pm\in \Gamma(M,E_i^\pm)$, $i\in\{0,j\}$
	\begin{align}
	\nabla^j_{Z_j}Z_0^{\pm} &=\Check p_j^{\pm}([Z_j,Z_0^{\pm}]) \\ 
	\nabla^j_{Z_0}Z_j^{\pm} &=p_j^{\pm}([Z_0,Z_j^{\pm}]) \\ 
	\nabla^j_{Z_0^{\mp}}Z_0^{\pm} &=\Check p_j^{\pm}([Z_0^{\mp},Z_0^{\pm}]) \\ 
	\nabla^j_{Z_j^{\mp}}Z_j^{\pm} &=  p_j^{\pm}([Z_j^{\mp},Z_j^{\pm}])  \\ 
	\nabla^j_{X}Z_0^{\pm} &=[X,Z_0^{\pm}] + C_{j}^{\pm}(X)Z_0^{\pm} \\
	\nabla^j_{X}Z_j &=[X,Z_j] 
	\end{align}

\end{remark} 

\begin{lemma}\label{lemageod}
    The geodesics of $\nabla^j$ which are tangent to $E^\pm$ are complete.
\end{lemma}
\begin{proof}
Associated with the connection $\nabla = \nabla^j$ we have the concepts of exponential map and normal neighbourhoods, which are analogous to the Riemannian case. On such a neighbourhood, the exponential map is give us the geodesic. Since $M$ is compact, there exists a constant $c>0$ such that for every tangent vector $Y$ of $M$ with $\|Y\|<c$ we can integrate the geodesic with initial condition $Y$ to a time equal or greater the one.\\

Now, suppose that we have $Y\in E^{+}$, this means that there exists $t_0$ such that for $t<t_0$ we have
\begin{align*}
\|d\phi(ta,\;\cdot\;)Y\|<c
\end{align*}
and through $d\phi(ta,\;\cdot\;)Y$ we can integrate the geodesic to a time equal or greater then one.  But the action $\phi$ is affine (with respect to $\nabla$), which means it transport geodesics, and therefore, we can integrate the geodesic through $Y$ to a time equal or greater then one. As $Y$ is arbitrary, the geodesics tangent to $E^+$ are complete.
\end{proof}

\subsection{Infinitesimal affine transformations and Killing fields}\label{subsecaffine}

Fix some point $v_0\in M$ with compact orbit. We define the Lie algebra $\hat\lieg'$ as the Lie algebra of germs at $v_0$ of infinitesimal affine transformations of the connections $\tilde \nabla^j$. That is, germs of vector fields $Y$ satisfying 
$$[\L_Y,\tilde\nabla^j_Z] = \tilde\nabla^j_{[Y,Z]}\enskip\enskip\forall \;j$$

Now we define the Lie algebras $\lieg'$ as germs of vector fields $Y\in \hat\lieg'$ that also satisfies:
\begin{align*}
    \begin{cases}
    \L_yd\alpha_j&=0\enskip\forall \;j\\
    \L_YX_s &= 0\enskip\forall \;s\\
    [Y,E_i^\pm]&\subset E_i^\pm\;\;\;\;i\in\{0,1\}
    \end{cases}
\end{align*}

Notice that the second condition implies the third. The following Lemma is straightforward.

\begin{lemma}
The Lie algebra $\lieg'$ is the Lie algebra of killing vector fields of the geometric structure defined in the previous subsection, and is contained in the the Lie algebra of germs of affine vector fields for the connection $\nabla^j$ for every $j$.
\end{lemma}
\begin{coro}\label{GromovRigidity}
The geometric structure $\sigma$ is Gromov-rigid.
\end{coro}
\begin{proof}
We just recall that linear connections are Gromov-rigid.
\end{proof}
\begin{coro}
The pseudo-group $\mathcal G$ of local automorphisms of $\sigma$ has an open, dense orbit $\Omega\subset M$
\end{coro}
\begin{proof}
We note that $\mathcal G$ contains the diffeomorphisms defined by the Anosov action. Moreover, the Anosov action preserves a volume form, and it is therefore topologicaly transitive, thus $\mathcal G$ has a dense orbit. The result follows from Gromov's open-dense orbit theorem.
\end{proof}

\begin{lemma}
    If $Y\in\lieg'$ then the $\alpha_s(Y)$ is constant for every $s$.
\end{lemma}
\begin{proof}
    Let us write $Y = Y_0 + Y_j \sum f_iX_i$, where $Y_0$ is tangent to $E_0$ and $Y_j$ is tangent to $E_j$. As $X_s$ preserves the splitting and $[X_s,Y]=0$, it follows that 
    $$X_s(f_i) = 0\enskip\enskip \forall i,s$$
    This means that the functions $f_i$ are constant along the orbits of the action. But the action is topologically transitive, and therefore $f_i = \alpha_i(Y)$ is constant.
\end{proof}

Consider the following Lie subalgebra:
$$\liek'_j= \{Z\in \lieg'\;;\;i_Zd\alpha_j = 0\;\;\alpha_1(Z) = \dots\alpha_l(Z) = 0\}$$
\begin{lemma}\label{splitinglemma}
The Lie sub-algebras $\liek_j'\subset \lieg'$ are ideals, and the induces a split short exact sequences:
\begin{align}\label{splitsequence}
    0\longrightarrow\liek'_j\longrightarrow\lieg'\stackrel{\pi^j}{\longrightarrow} \underbrace{\lieg'\slash \liek_j'}_{ \lieg_j}\longrightarrow 0
\end{align}
\end{lemma}
\begin{proof}
    That $\liek'_j$ are ideals is straightforward.  Consider the splitting $TM = T\phi^0\oplus T\phi'\oplus E_0\oplus E_j$, where $T\phi^0$ is tangent to the action generated by $\{X_1,\dots,X_l\}$ and $T\phi'$ by the action $\{X_{l+1},\dots,X_k\}$.
    According to this splitting, we consider the projections:
    \begin{align*}
    		 P_j&:TM\to T\phi'\oplus E_j\\
	    \Check P_j&:TM\to T\phi^0\oplus E_0
    \end{align*}
    
    Clearly $Id = \Check P_j + P_j$.   Let $Y\in \lieg'$, let us show that $P_j(Y)$ and $\Check P_j(Y)$ also belongs to $\lieg'$. First, we remark that if $Y\in\lieg'$ then, in particular, $Y\in\tilde\lieg'$ and from our definition of rigid $k$-geometric actions, as $[Y,X_s]=0$, then $[\Check P_j(Y),X_s]=0$ and thus we have $\Check P_j(Y)\in \tilde\lieg'$ and thus, $P_j(Y) = Y - \Check P_j(Y)\in\tilde\lieg'$.
    
    Now, let us check that $Y'=P_j(Y)$ belongs to $\lieg'$. 
    \begin{itemize}
    \item $\L_{Y'}d\alpha_j = d\circ i_{Y'}(d\alpha_j) = 0$ as $T\phi'\oplus E_j \subset\ker d\alpha_j$.
    \item As $X_s$ preserves the splitting, $[X_s,E_i]\subset E_i$ and thus
    $$0 = [Y,X_s] = [P_j(Y) + \Check P_j(Y),X_s] = [P_j(Y),X_s]+ [\Check P_j(Y),X_s]$$
    implies that both terms are zero.
    \end{itemize}
    
    Now, it is easy to see that $\liek_j' = P_j(\lieg')$. Now, let us see that $\Check P_j:\lieg'\to\lieg'$ is a lie algebra homomorphism. We recall that $\liek_j'$ is an ideal and thus $\Check P_j([A,\liek'])=0$. We have
    \begin{align*}
        \Check P_j([A,B]) &= \Check P_j([\Check P_jA + P_jA,\Check P_jB+ P_jB])\\
        &=\Check P_j\big(\Check P_jA,\Check P_jB]+[\Check P_jA,P_jB]+[P_jA,\Check P_jB]+[P_jA,P_jB]\big)\\
        &=\Check P_j[\Check P_jA,\Check P_jB]
    \end{align*}
    Now,as $T\phi^0\oplus E_0$ is integrable, $[\Check P_jA,\Check P_jB]$ is tangent to $T\phi^0\oplus E_0$, and we obtain $\Check P_j([A,B]) = [\Check P_jA,\Check P_jB]$.
    
    Finally, we observe that we can identify $\lieg'\slash \liek_j'$ 
     with $\Check P_j(\lieg')$ as we desired. In particular, as the sequence splits, we can write $\lieg' = \lieg_j\ltimes\liek'_j$ and the map $\pi^j:\lieg'\to\lieg_j$ is actually the projection $\Check P_j$.
\end{proof} 
\begin{remark}\label{remsplit}
If we write $\lieg' = \lieg_j\ltimes\liek'_j$ and $Id = \Check P_j + P_j$ the projections on the first and second factors, as above, it is clear that, for $s = 1,\dots, l$ we have $X_s  = \Check P_j(X_s) + P_j(X_s) = \Check P_j(X_s)$, that is, for every $j$, if we write $\lieg' = \lieg_j\ltimes\liek'_j$ then $X_s\in \lieg_j$. In other words, let $\sigma_j:\lieg_j\to\lieg'$ be the morphisms which splits the sequence (\ref{splitsequence}), then $X_s = \sigma_j\circ \Check P_j(X_s)$.

\end{remark}

Finally, we define the following two Lie sub algebras:

\begin{align*}
    \lieh' &= \{Y\in\lieg'\;\; Y_p = 0\}\\
    \lieh_j &= \Check P_j(\lieh')\subset\lieg_j
\end{align*}

Our end goal is to model our manifold after $G'\slash H'$ where $G'$ and $H'$ are Lie groups with Lie algebras $\lieg'$ and $\lieh'$. Naturally, some questions arises: Are there groups $G'$ and $H'$ such that $H'$
 is closed in $G'$? The following Lemma shows that we can reduce this question to the Lie groups $G$ and $H$ corresponding to $\lieg$ and $\lieh$.
 
 \begin{lemma}\label{closedimpliesclosed}
 Let $G'$ and $G$ be the simple connected Lie groups with Lie algebras $\lieg'$ and $\lieg$, and let $H'\subset G'$ be the connected subgroup with Lie algebra $\lieh'$. Let $\pi:G'\to G$ be the unique Lie group homomomorphism associated with the projection map $\lieg'\to \lieg'\slash \liek'=\lieg$, and let $H = \pi(H')$. Suppose that $H\subset G$ is closed, then $H'\subset G'$ is closed. 
 \end{lemma}
 \begin{proof}
 As $\pi$ is continuous, $\hat H = \pi^{-1}(H)\subset G'$ is a closed subgroup. Moreover, if $K'\subset G'$ is a connected subgroup with Lie algebra $\liek'$, then $K'$ is a normal subgroup and $\hat H = H' \cdot K'$. Now, let $H\ni h_j\to z$ be a converging sequence. As $\hat H$ is closed, $z = h\cdot k$ for $h\in H'$ and $k\in K'$.
 
 We can suppose that $z$ is small\footnote{Notice that for sufficiently large $N$ the sequence $h_N^{-1}h_j$ is bounded near the identity, and thus, we can always suppose, by translating the sequence, that it converges to somwhere near the identity.}, then $ h_j = \exp a_j$ for $a_j \in \lieh'$ and $h = \exp a$ and $k = \exp b$ for $a\in\lieh'$ and $k\in\liek'$. Now, we use Baker-Campbell-Hausdorff formula, and use the fact that $\liek'$ is an ideal to write:
 $$z =\exp(a)\exp(b) = \exp(a + \tilde k)\enskip\enskip\text{for some }\tilde k\in \liek'$$ 
 for some $\tilde k\in \liek'$. Because the exponential is a diffeomorphism near the identity, $h_j\to z$ implies $a_j\to a + \tilde k$. But $\lieh'$ is a closed subspace, and thus, $a + \tilde k\in \lieh'$, that is, $z\in H'$. 
 \end{proof}

We shall now, develop the tools to study the Lie algebras we defined above. Our main tool is given by the following definition and Lemma, which will allow us to do some computations with the the Lie algebra $\lieg'$, by making use of the fact that it also an algebra of infinitesimal affine transformations for the connections $\nabla^j$

\begin{definition}
Let us denote $V_0 = T_{v_0}M$.
For each $j$, the connection $\nabla^j$ induces a a bracket operation on $End(V_0)\times V_0$ given by:
$$[(A,a),(B,b)]_j:= ([A,B] +R^{\nabla^j}(a,b),T^{\nabla^j}(a,b) + Ab - Ba)$$
\end{definition}
The bracket $[\;\cdot\;,\;\cdot\;]_j$ doesn't make $End(V_0)\times V_0$ a Lie algebra, but it allows us to embed the Lie algebra $\lieg'$ on $End(V_0)\times V_0$. The classical (Kobayashi-Nomizu) lemma about killing fields has the following adaptation:
\begin{lemma}\label{repres}
For each $j$, the map
\begin{align*}
    \theta_o^j:\lieg'&\to End(V_0)\times V_0\\
    Y&\mapsto\big((L_Y - \nabla^j_Y)_{v_0}, Y_{v_0}\big)
\end{align*}
is a monomorphism and it preserves the bracket.
\end{lemma} 
\begin{proof}
To prove that $\theta_o$ is injective we shall show that the germ of a Killing vector field $X$ at $o$ depends only on $(A_X)_o$ and $X_o$. To see this, consider a smooth curve $\gamma(t)$ on $M$ passing through $o$. We denote by $a(t) = A_{X(\gamma(t))}$, $v(t) = \dot\gamma(t)$ and $x(t) = X(\gamma(t))$. From $(\nabla_Y(A_X))Z = R(X,Y)Z$\footnote{This is a classical result, Kobayashi-Nomizu}, and $A_XY = -\nabla_YX - T(X,Y)$, it follows that
\begin{align*}
\nabla_{v(t)}a(t) &= R(x(t),v(t))\\
\nabla_{v(t)}x(t) &=- T(x(t),v(t)) -a(t)v(t)
\end{align*}

We have therefore a Cauchy problem and thus, the germ of $X$ at $o$ depends only on the initial data: $(A_X)_o$ and $X_o$.\\

Straightforward computations shows that the bracket is preserved.
\end{proof}

Let $\mathcal H'$ denote the the group of local automorphism of the geometric structure, which preserves our base point $p$. It is clear that the Lie algebra of $\mathcal H'$ is precisely $\lieh'$. The following Lemma connect this Lie group with the embedding $\theta_o^j$.
\begin{lemma}\label{injectiv}
Consider the map $i:\mathcal H'\to Aut(V_0)$ given by $i(\varphi) = d\varphi_{v_0}$. Then this map is a finite covering onto it's image and its differential coincides with $\theta_o^j$ restricted to $\lieh'$.
\end{lemma}
\begin{proof}
	To prove the second assertion, we consider $Y\in \lieh'$ and $\psi^t: = \exp(tY)$. Let $Z$ be a vector field around $v_0$. As $\psi^t$ fixes $v_0$, the parallel transport in this orbit reduces to the identity, and therefore, $(\nabla_YZ)_{v_0} = 0$. We have
	\begin{align*}
	di_e(Y) &= \frac{d}{dt}\bigg|_{t=0} i(\exp(tY)) = \frac{d}{dt}\bigg|_{t=0} (\psi^t_*)
	\end{align*}
	but, for a vector field $Z$ around $v_0$
	\begin{align*}
	\lim_{t\to 0}\frac{1}{t}(\psi^t_*Z -Z) &= \lim_{t\to 0}\frac{1}{-t}(\psi^{-t}_*Z -Z) \\
	&= -(\L_YZ)_{v_0} = (\nabla_YZ)_{v_0} - (\L_YZ)_{v_0} = (A_Y)_{v_0}Z_{v_0}
	\end{align*}
	and thus,
	$$di_e(Y) = (A_Y)_{v_0}$$ 
	as we wanted.\\
	
	To see that $i$ is a covering map, we just notice that $\theta_0^j$ is injective, (Lemma \ref{repres}). It remains to prove that his covering is finite.
	
	From the definition of Gromov-rigidity, the map
	\begin{align*}
	    H'&\to Gl^r(T_{v_0}M)\\
	    h&\mapsto j^r_{v_0}(h)
	\end{align*}
	is injective for some jet group $Gl^r(T_{v_0}M)$. We recall that the exact sequence
	$$1\to\ker(p)\to Gl^r(T_{v_0}M)\stackrel{p}{\to} Gl(T_{v_0}M)\to 1$$
	splits and thus, induces a splitting on the exact sequence:
	$$1\to\ker(p)\to H'\stackrel{p}{\to} i(H')\subset Gl(T_{v_0}M)\to 1$$
	In particular, we can write
	$$H' = i(H')\ltimes \Gamma$$
    Where $\Gamma$ is discrete. But $H'$ is algebraic, and thus has only a finite number of connected components. It follows that $\Gamma$ is finite.
\end{proof}

\begin{remark}
Notice that while the map $\theta_o^j$ does depends on the choice of $j$ and associated connection, the map $i$ doesn't, and thus, the restriction of $\theta_o^j$ to the Lie algebra $\lieh'$ doesn't depends on the choice of $j.$
\end{remark}

Let  $\varphi_t$ be the flow of $X$ given by an Anosov element. As the Anosov elements form an open cone on $\R^k$, we have some freedom to choose this element. We shall chose an element, such that $v_0$ is a periodic point of $\varphi_t$ of period $t_0$. Consider the map $(d\varphi_{t_0})_{v_0} \in \mathcal H'\subset Aut(V_0)$. As $\mathcal H'$ is algebraic, it has a finite number of connected components, and thus, there exists $n_0\in \N$ such that $(d\varphi_{nt_0})_{v_0}$ is in the connected component (of $\mathcal H'$) of the identity. We set $l_0 = (d\varphi_{nt_0})_{v_0}$, and $L_0$ the logarithm of the hyperbolic part of the Iwasawa decomposition of $l_0$\footnote{While it is not true, in general, that the logarithm of an element of a Lie group exist, the logarithm map is well defined on the hyperbolic part of the Iwasawa decomposition}.  In particular $\exp L_0$ have positive eigenvalues.
\begin{lemma}
	The following assertions are valid:
	\begin{enumerate}
		\item{}$\mathcal H'$ is an algebraic subgroup of $Aut(V_0)$
		\item{}$L_0\in \lieh'$
		\item{}We can choose an iterate $l_0^r$, $r\geq 0$ of $l_0$ such that the eigenvalues of the corresponding $L_0$ (that is, the logarithm of the hyperbolic part of $l_0^r$) on $E^{+}_{v_0}$ (resp. $E^{-}_{v_0}$) are strictly negative (resp. positive).
	\end{enumerate}
\end{lemma}
\begin{proof}

	\begin{enumerate}
		\item{} From the demonstration of the open-dense orbit theorem, it follows that 
		$$\mathcal H' = \Aut^{loc}_{v_0v_0}(\sigma) = \Aut^r_{v_0v_0}(\sigma)$$
		for some $r$. But $\Aut^r_{v_0v_0}(\sigma)$ which is algebraic by construction.
		\item{} As $\mathcal H'$ is algebraic, it contains both the hyperbolic and the elliptic part of its elements (Helgason \cite{helgason}, IX.$\S$7 Lemma 7.1,  or Seco et al. \cite{Seco}.).
		\item{} This will follows from the Anosov property of the action. In fact, as $L_0\in \lieh'$, we use the identification $\theta_o$ and write $L_0 = (\tilde L_0,0)$, where $\tilde L_0 = (\nabla_{L_0} - \L_{L_0})_{v_0}$. However, as we have seen (Lemma \ref{injectiv}), the restriction of $\theta_o$ to $\lieh'$ coincides with the differential of $H'\ni h \mapsto dh_{v_0}\in GL(V_0)$. 
		
		Thus, $\tilde L_0$ acts on $Z_0 \in E_{v_0}$ precisely as the (hyperbolic part of the) differential of an Anosov element. Let us call this differential $T$.
		
		While it may not be true that $$\|T_{|_{E_{v_0}^+}}\|\leq 1$$
		It is true that
		$$\|T^r_{|_{E_{v_0}^+}}\|\leq 1$$
		for some $r$. Thus, $T^r$ restricted to $E_{v_0}^+$ have eigenvalues of module less then one, and therefore it's hyperbolic part restricted to $E_{v_0}^+$
		have eigenvalues on the open interval $]0,1[$, and thus its logarithm have strictly negative eigenvalues.
		   
	\end{enumerate}
\end{proof}

           \begin{lemma}\label{curvlemma}
The curvature $K^j$ of $\nabla^j$ satisfies:
\begin{enumerate}
    \item $K(X_s,X_l)W_0^+ = 0$
    \item $K(X_s, Y_j^\pm)W_0^+ = 0$
    \item $K(Y_j^-,Z_j^-)W_0^+ = 0$
    \item $K(Y_j^+,Z_j^+)W_0^+ = 0$
\end{enumerate}
\end{lemma}
\begin{proof}

\begin{description}
\item[(1)] We notice that for $W = W_0^+$
\begin{align*}
    \nabla_{X_s}\nabla_{X_l}W &= \nabla_{X_s}\big([X_l,W] + C_j^+(X_l)W\big)\\
    &= [X_s,[X_l,W] + C_j^+(X_l)W] + C_j^+(X_s)\big([X_l,W] + C_j^+(X_l)W\big)\\
    &=[X_s,[X_l,W]] + C_j^+(X_l)[X_s,W] + C_j^+(X_s)[X_l,W] + C_j^+(X_s)C_j^+(X_l)W
\end{align*}
and thus, 
$$\nabla_{X_s}\nabla_{X_l}W -\nabla_{X_l}\nabla_{X_s}W = [X_s,[X_l,W]] -[X_l,[X_s,W]] = [W,[X_s,X_l]] = 0$$

Finally, we have
\begin{align*}
    K(X_s,X_l)W:= \nabla_{X_s}\nabla_{X_l}W -\nabla_{X_l}\nabla_{X_s}W - \nabla_{\underbrace{[X_s,X_l]}_{=0}}W=0
\end{align*}
\item[(2)] As $X_s$ preservers the splitting $TM = T\phi\oplus E_0^+\oplus E_0^-\oplus E_j^+\oplus E_j^-$, it follows, that, for any vector field $Z$ we have
$$\Check p_j^+([X_x,Z]) = [X_s,\Check p_j^+(Z)]$$
and thus, writing $X = X_s$, $Y = Y_j^\pm$ and $W = W_0^+$ we have the following computations
\begin{align*}
    K(X_s, Y_j^\pm)W_0^+ &= \nabla_X\nabla_YW -\nabla_Y\nabla_XW - \nabla_{[X,Y]}W\\
    &=[X,\nabla_YW] +C_j^+(X)\nabla_YW - \Check p_j^+[Y,\nabla_XW] - \Check p_j^+[[X,Y],W]\\
    &=[X,\Check p_j^+[Y,W]] + C_j^+(X)\Check p_j^+[Y,W] - \Check p_j^+[Y,[X,W] + C_j^+(X)W] -\Check p_j^+[[X,Y],W]\\
    &=\Check p_j^+\big([X,[Y,W]] -[Y,[X,W]] - [[X,Y],W]\big) = 0
\end{align*}

\item[(3)]
Let us write the Bianchi's identity:
$$\mathfrak S(K(A,B)Z) = \mathfrak S(T(T(A,B),Z)) + \nabla_ZT(A,B)$$
where $\mathfrak S$ denotes the cyclic sum. It is clear that if $A,B$ are tangent to $E_j^+\oplus E_j^-$, then the last term vanishes. Thus
$$\mathfrak S(K(Y_1^-,Z_1^-)W_0^+) = T(T(Z_1^-,W_0^+),Y_1^-) + T(T(W_0^+,Y_1^-),Z_1^-)$$

Now, from
\begin{align*}
    T(E_1^-,E_0^+) = \Check p_j^+([E_1^-,E_0^+]) - \nabla_{E_0^+}E_1^- -[E_1^-,E_0^+]
\end{align*}
and $Id = \Check p_j^+ + \Check p_j^- +p_j^+ + p_j^- + \sum_{s}\alpha_s(\;\cdot\;)x_s$,
it follows that $\Check p_j^+(T(E_j^-,E_0^+))=0$, that is $T(E_0^+,E_j^-)$ is tangent to $E_0^-\oplus E_j^+\oplus E_j^-\oplus T\phi$.

Simple computations, shows that $T(X_s,Z^\pm) = C_j^\pm(X_s)p_0^\pm Z^\pm$, moreover, from the rigidity condition of our geometric structure, it follows that for $A,B$ tangent to $E_j = E_j^+\oplus E_j^-$ we have $T(A,B)$ is tangent to $T\phi$.

Thus, if we write $T(Z_j^-,W_0^+)= A_0^- + A_j + A_\phi$ where $A_0^-$ is tangent to $E_0^-$, $A_j$ is tangent to $E_j^+\oplus E_j^-$ and $A_\phi$ is tangent to $T\phi$, then 
\begin{align*}
    T(T(Z_j^-,W_0^+),Y_j^-) &= T(A_0^- + A_j + A_\phi,Y_j^-)\\
    &=\underbrace{T(A_0^-,Y_j^-)}_{tangent\; to\; E^-} + \underbrace{T(A_j,Y_1^-)}_{tangent\;to\; T\phi} + \underbrace{T( A_\phi,Y_j^-)}_{tangent\; to\; E_j^- }
\end{align*}
thus $ T(T(Z_j^-,W_0^+),Y_j^-)$ and $T(T(W_0^+,Y_j^-),Z_j^-)$ are tangent to $E^-\oplus T\phi$. Because the curvature also preserves the splitting, we obtain
$$K(Y_j^-,Z_j^-)W_0^+ = 0$$

\item[(4)] If we write $Y = Y_j^+$, $Z = Z_j^+$ and $W = W_0^-$ we have
\begin{align*}
    K(Y_j^+,Z_j^+)W_0^+ &=\nabla_Y\nabla_ZW -\nabla_Z\nabla_YW - \nabla_{[Y,Z]}W\\
    &=\Check p_j^+\big([Y,\Check p_j^+[Z,W]] - [Z,\Check p_j^+[Y,W]] - [[Y,Z],W]\big)\\
    &=\Check p_j^+\big([Y,\Check p_j^+[Z,W]] - [Z,\Check p_j^+[Y,W]] - [Y,[Z,W]] + [Z,[Y,W]]\big)\\
    &=\Check p_j^+\big([Y,(\Check p_j^+ - Id)[Z,W]] - [Z,(\Check p_j^+ - Id)[Y,W]]\big)\\
\end{align*}
Now, for $E^+ = E_0^+\oplus E_j^+$ we wirte $Id = \Check p_j^+ + p_j^+$, moreover, $E^+$ is integrable, and thus, $[Z,W]$ and $[Y,W]$ are tangent to $E^+$. 
\begin{align*}
    K(Y_j^+,Z_j^+)W_0^+
    &=\Check p_j^+\big([Y,(\Check p_j^+ - Id)[Z,W]] - [Z,(\Check p_j^+ - Id)[Y,W]]\big)\\
    &=\Check p_j^+\big([Y,-p_j^+[Z,W]] - [Z,-p_j^+[Y,W]]\big)
\end{align*}
Now, $T\phi^j\oplus E_j$ is integrable, and thus $[Y,-p_j^+[Z,W]] - [Z,-p_j^+[Y,W]]$ is tangent to $T\phi^j\oplus E_j$. In particular
$$p_0^+([Y,-p_j^+[Z,W]] - [Z,-p_j^+[Y,W]])=0$$

\end{description}

\end{proof}

  \section{About the Lie algebra $\lieg$}\label{tecsec}
            
Recall that on the construction of the connections $\nabla^j$ we had some choices of linear maps $C_j^\pm:\R^k\to\R$. From here on we shall suppose that $C_j^\pm(X_i)= \delta_{ij}\ss_j^\pm$ for some non zero constants $\ss_j^\pm$. Latter on we shall choose this constants precisely.
\begin{theorem}\label{reductible}
    Let $\Ii' = Span_\R \{X_1,\dots,X_l\}\subset\lieg'$ and $\Ii_j = \Check p_j(\Ii')\subset\lieg_j$. Then the Lie algebra $\lieg_j$ is reductive and its centre $\Ii_j$\footnote{We recall that, from Remark \ref{remsplit},  if we write $\lieg' = \lieg_j\ltimes\liek'_j$, then $\Ii_j = \Ii'$} 
\end{theorem}
This theorem will follow from the following lemma:
\begin{lemma}\label{nilradicallemma}
	The nil-radical of $\lieg_j$ is $\Ii_j$.
\end{lemma}

\begin{proof}[Proof of Theorem \ref{reductible}]
    To avoid cluttering the notation, we shall omit the index $j$. From here on, we shall denote $\Ii = \Ii_j$ and $\lieg = \lieg_j$ 
	First observe that the lemma implies that $\Ii$ is in fact the centre of $\lieg$, just notice that the center is in fact a nilpotent ideal of $\lieg$.
	Consider a solvable ideal $\liea$ of $\lieg$. Then there exists ideals $\liea_1,\dots,\liea_l$ such that $[\liea_j,\liea_j]\subset \liea_{j+1}$ and
	$$\liea = \liea_0\geq \liea_1\geq\cdots\geq \liea_r\geq \liea_{r+1} =0$$
	
	In this setting, $\liea_r$ is an abelian ideal, and therefore is nilpotent. From Lemma \ref{nilradicallemma}. It follows $\liea_r\subset \Ii$.
	
	Now, $[\liea_{r-1},\liea_{r-1}]\subset\liea_r\subset\Ii$ and thus
	$$[[\liea_{r-1},\liea_{r-1}],\liea_{r-1}]\subset [\Ii,\liea_{r-1}]=0$$
	that is, $\liea_{r-1}$ is a nilpotent ideal of $\lieg$ and therefore $\liea_{r-1}\subset\Ii$. Proceeding by induction, we obtain $\liea\subset \Ii$ and thus $\Ii$ is in fact the radical of $\lieg$, which is therefore reductive.
\end{proof}

\begin{proof}[Proof of Lemma \ref{nilradicallemma}]
We follow the convention of the last Lemma and omit the index $j$.

	Notice that each $X_s$ belongs to the center of $\lieg'$ and thus $\Check p_j(X_s)$ belongs to the center of $\lieg$, that is, $\Ii\subset Z(\lieg)$. Take a nilpotent ideal $\liej\subset \lieg$, and let $\liej' =\Check p_j^{-1}(\liej)$. As $\exp L_0$ is hyperbolic, this means that $L_0$ has real eigenvalues and we can write

\begin{align*}
\lieg' = \bigoplus_{i\in\R}\lieg'_i\enskip;\enskip 
\lieg = \bigoplus_{i\in\R}\lieg_i\enskip;\enskip
\liej' = \bigoplus_{i\in\R}\liej'_i\enskip;\enskip
\liej = \bigoplus_{i\in\R}\liej_i\enskip;\enskip  
V_0:=T_{v_0}M = \bigoplus_{i\in\R}V_{0,i}
\end{align*}
where $\lieg'_i$ (respectively $\liej_i$ and $V_{0,i}$) is the (generalized) eigenspace associated with the eigenvalue $i$ of the action of $L_0$ on $\lieg'$ (respectively $\liej$ and $V_0$).\\

Remember that we are identifying $\lieg'$ with it's image under the map $\theta^j_o:\lieg'\to \End(V_0)\oplus V_0$. Also remember that $\lieh'$ is identified with its image (via $\theta_o$) on $\End(V_0)$, and thus, we identify $L_0 = (L_0,0)$. 
We shall denote by $\overline {(A,a)}$ the class of a $(A,a)\in\lieg'$ in $\lieg$. Notice, that for an element $\overline{(A,a)}$ we can choose a representative $(A',a')$ such that $a'$ is tangent to $T\phi^0\oplus E_0^+\oplus E_0^-$, that is, $p_j^\pm(a')=0$.

Now, for $i\neq0$ let $(A,y)\in \liej'_i$, then, for every $(A',y')\in \lieg'_{-i}$, we have:
\begin{align*}
(A'',y'') =: [(A,y),(A',y')]\in \liej'_0
\end{align*}
Notice that $\liej'_0\subset \lieg'_0$, and
\begin{align*}
\lieg'_0 &= \{(Q,q)\in \lieg'; [L_0,(Q,q)]=0 \}\\
&= \{(Q,q)\in \liej; ([L_0,Q], L_0q)=0 \}
\end{align*}

but, the eigenvalues of $L_0$ restricted to $E^{+}\oplus E^{-}$ are strictly non zero, and thus, $\lieg'_0\subset \lieh'\oplus\Ii$. Now, as $(A,y)\in \liej_i$ for $i\neq 0$, we have that $L_0y=iy$, and thus $y\in E^{+}\oplus E^{-}$, it follows that, $\alpha_j(y)=0$. And analogous for $y'$, that is, $(A',y')\in \lieg'_{-i}$, and thus, $L_0y' = -iy'$ and thus, $y'\in E^{+}\oplus E^{-}$, therefore $\alpha_j(y')=0$. It is also clear that either $y = p^+y$ and $y'  = p^-y' $ or $y = p^-y$ and $y'  = p^+y' $, and some computations shows that either $Ay'  = p^+Ay' $ and $A'  y = p^-A'  y$ or $Ay'  = p^-A y' $ and $A'  y = p^+A'  y$. Thus, from $(A'',y'') =: [(A,y),(A',y')]\in \liej_0$, it follows that $y''$ belongs to $T\phi$, that is
\begin{align*}
    y'' = \sum_s\alpha_s(y'')= \sum_s\alpha_s\bigg(\!\!\!\!\overbrace{\nabla_yy'   - \nabla_{y' } y}^{\text{tangent to }E^+\oplus E^-}\!\!\!\! - [y,y' ] +\!\!\!\!\!\!\!\!\!\underbrace{Ay'-A' y}_{\text{tangent to }E^+\oplus E^-}\!\!\!\!\!\!\!\!\bigg) = -\sum_s\alpha_s([y,y'])X_s
\end{align*}

Therefore, 
\begin{align*}
(A'',y'') =([A,A'] - R(y,y'), -\sum_s\alpha_s([y ,y'])X_s)
\end{align*}

Now, consider $\theta_o(X_s) = ((A_{X_{s}})_{v_0},X_{s,v_0})$. As
\begin{align*}
A_{X_{s}}Y&=\L_{X_{s}}Y -\nabla_{X_{s}}Y \\
&=\L_{X_{s}}Y -[X_{s},Y] - S_{sj}(Y) =-S_{sj}(Y)  
\end{align*}
where $$S_{sj}(Y) = C_j^+(X_s)\Check p_j^+(Y) +   C_j^-(X_s)\Check p_j^-(Y)$$
we have that $\theta_o(X_s) = (-S_{sj,v_0},X_{s,v_0}) = (-S_{sj},X_s)$. Thus,

\begin{align*}
(A'',y'') +\sum_s\alpha_s([y,y'])\theta_o(X_s) &=
\sum_s([A,A'] - R(y,y') -\alpha_s([y,y'])S_{sj},0)
\end{align*}

That is, 
$$(A'',y'')Mod\;\Ii' = \sum_s([A,A'] - R(y,y') -\alpha_s([y,y'])S_{sj},0)$$

Now, $\mathfrak j$ is a nilpotent ideal, and therefore the adjoint of its bracket has null trace for every invariant subspace stable (Foulon, P. and Labourie, F. \cite{FL}. Proof of Proposition 3.7)
by $\mathfrak j$. Moreover, from identifying $\lieg'\slash\lieh' = T_{v_0}M$, we also identify $$\lieg\slash\lieh= (T\phi^0)_{v_0}\oplus(E_0^+)_{v_0}\oplus(E_0^-)_{v_0}$$ and we obtain:
$$Tr|_{E_{0}^\pm}(A'',y'') = 0$$

Notice that $\Ii'$ is contained in the centre, and therefore it doesn't contribute to the trace, and thus
$$Tr|_{E_{0}^\pm}(A'',y'')Mod \;\Ii' = 0$$

From the linearity of the trace, we have
$$0 =Tr|_{E_{0}^\pm}[A,A'] + Tr|_{E_{0}^\pm}R(y,y') +\sum_s Tr|_{E_{0}^\pm}\alpha_s([y,y'])S_{sj}$$
The first term is clearly zero, while the trace of a projection is given by the dimension of the subspace it projects into, and thus $Tr_{E_{0}^\pm}S_{sj} = qC_j^\pm(X_s)$.

Now, consider the determinant bundles $\Lambda_0^\pm := \Lambda^q(E_0^\pm)^*$. The connections $\nabla^j$  induces a natural connection $\hat\nabla^j$ on $\Lambda_0^\pm$. The curvature of $\Lambda_0^\pm$ is given by:
$$K^{\hat\nabla^j}(Y,Z) = Tr(K^{\nabla^j}(Y,Z)|_{E_0^\pm})$$
Suppose that $Y,Z$ are tangent to $E_0^+\oplus E_0^-$, then, because $(d\alpha_j)_{|_{E_0^+\oplus E_0^-}}$ is non degenerate, there exists a linear map $B^j$ such that $K^{\hat\nabla^j}(Y,Z)= d\alpha_j(B^jY,Z)$.
\begin{lemma}\label{nilpotentB}
    The map $B^j$ is nilpotent.
\end{lemma}
For now, we accept the truth of this Lemma. We shall provide the proof in the next section. Back to our computations, we are interested in $\overline{(A,y)}$, thus we can choose $(A,y)$ and $(A',y')$ such that $y$ and $y'$ are tangent to $E_0^-\oplus E_0^-$.

We obtain
\begin{align*}
0 = &d\alpha_j(B^j y,y') + \sum_{s}\alpha_s([y,y'])qC_j^\pm(X_s)^+\\
&=d\alpha_j(B^j y,y') - q\sum_{s}C_j^\pm(X_s)^\pm d\alpha_s(y,y')
\end{align*}

But, we have chosen $C_j^\pm$ such that $C_{j}^\pm(X_s)= \delta_{sj}\ss_j^\pm$ for non zero $\ss_j^\pm$, and thus

\begin{align*}
0 = &d\alpha_j((B^j + q\ss_j^+ )y,y')  
\end{align*}

As this is true for every $y\prime$, it follows $(B^j + q\ss_j^+)y$. But  $\ss_j^+\neq0$ and $B^j$ is nilpotent (and thus doesn't have a  non zero eigenvalue), therefore  $y = 0$.

We conclude that $\liej_i\subset \lieh$\footnote{We proved that for $\overline{(A,y)}\in \liej_i$, there exists a representative $(A,y)$ such that $y=0$, that is $(A,y)\in\lieh'$, and thus $\overline{(A,y)}\in \Check p_j(\lieh')=\lieh_j =\lieh$}. Thus, $$\liej\subset\lieh\oplus\Ii$$

Now, $\lieh$ can't contain an ideal of $\lieg'$\footnote{Let $\overline{(A,0)}\in \lieh$, if we chose any $\overline(B,b)\in \lieg$ such that $b\in V_0^+$ and $Ab\neq 0$, we have $$[(A,0),(B,b)] \not\in\lieh'$$
}, thus, the projection $\liej\to \lieh$ is zero and we obtain $\liej\subset \Ii$ as wanted.
\end{proof}

        \section{Proof of Lemma \ref{nilpotentB}}\label{Nilsec}
            The goal of this section is to prove Lemma \ref{nilpotentB}. 

The proof follows closely the proofs of Benoist-Foulon-Labourie and, where no major alterations are needed, we omit the proof as needed.\\

As we mentioned before, we consider the determinant bundle $\Lambda_0^\pm := \Lambda^q(E_0^\pm)^*$ with induced connection $\hat\nabla^j$ and curvature $\Omega^j:=K^{\hat\nabla^j}$. To avoid cluttering the notation, we omit the superscript $j$ for the duration of this section, still, we reserve the letter $j$, for the particular choice of connection $\nabla^j$. We also abuse the notation and write simply $\hat\nabla^j = \nabla$

Before proving this, lets build a primitive of $\Omega^j$. Let $\zeta =\zeta^j$ be a section of $\Lambda_0^+$. If $\Lambda_0^+$ is trivial, we can assume this section to be never vanishing, otherwise, we take a section $\zeta$ modulo sign (that is, a section of $|\Lambda_0^+|$\footnote{The fiber of $|\Lambda_0^+|$ is actually $\R\slash\{\pm Id\} = [0,+\infty)$}).\\

Now, remember that the connection form $\beta$ of a vector bundle $E$ (with respect to a  local frame $s$) is given by $$\nabla^E s = s\cdot \beta$$ where $\beta$ can be understood as a $End (E)$ valued form. On the case at hand, our bundle is uni-dimensional, and our local frame is just a single section $\zeta$. We define thus the (real valued) one form $\beta = \beta^{j}$ by
\begin{align*}
\nabla_Z\zeta^j = \beta^j(Z)\zeta^j
\end{align*}

Finally, observe that the curvature form $\Omega$ of $E$ (with respect to the local frame $s$) is given (in terms of the connection form $\beta$), by:
$$\Omega = d\beta + \beta\wedge\beta$$

On the case at hand, $\beta$ is actually a "normal" real valued one form\footnote{More technically, as our bundle is one dimensional, then our connection form has values on a one dimensional Lie algebra, which is therefore, abelian, and thus, $\beta\wedge\beta(A,B) = [\beta(A),\beta(B)]=0$}, and therefore, $\beta\wedge\beta=0$, thus, 

$$\Omega= \Omega^j = d\beta^{j}= d\beta$$

\begin{lemma}\label{intbeta}
	Consider the connection form $\beta^{j}$ and curvature form $\Omega$ on $\Lambda_0^{+}$. Consider $\Theta$ the invariant form in the definition of rigid $k$-geometric actions \ref{rigidkgeo} and let us write $dM_j$ for the volume form
	$\alpha_1\wedge\dots\wedge\alpha_k\wedge d\alpha_j^n\wedge\Theta$.
	Then, for appropriate choice of linear maps $C_j^\pm $ (involved in the definition of the connection: Lemma \ref{buildconnec}), we have
	$$\int_M\beta(X_j)dM_j=0$$
\end{lemma}
Before we prove this lemma, it will be necessary to give a new definition. This definition will be used to choose the appropriate maps $C_{j}^\pm$
\begin{definition}\label{entropy}[Entropy of a invariant subbundle]
	Let $F$ be a	 sub-bundle of $TM$ that is invariant by a flow $\psi_t$. Choose a volume form $dx$ of $M$ and a never vanishing section $\zeta$ of the determinant bundle of $F$, that is of the volume forms of $F$ (if necessary, quotient the bundle modulo sign). We define, for $t\in\R$
	$$a_t = a_t(F) = \int_M\log|\det(d\psi_t)_F|dx$$
	
	where the determinant is taken with the help of the section $\zeta$. Now, observe that for $t,s\in\R$, we have:
	\begin{align*}
	a_{t+s} &= \int_M\log|\det(d\psi_{t+s})_F|dx\\
	&=\int_M\log|\det(d(\psi_t\circ\psi_s))_F|dx = \int_M\log|\det(d\psi_td\psi_s)_F|dx\\
	&=\int_M\log|\det(d\psi_t)_F\det(d\psi_s)_F|dx =\int_M\log|\det(d\psi_t)_F|+\log|\det(d\psi_s)_F|dx = a_t + a_s
	\end{align*}
	
	This means that $a_t$ is a continuous one parameter subgroup of $\R$. Therefore, there exists $\ss = \ss(F)\in\R$ such that
	$$a_t = \ss\cdot Vol(M)\cdot t$$
	where $Vol(M) = \int_Mdx$. Notice that $\ss$ is does not depend on the choice of the section. The quantity $\ss$ is called the entropy of the sub-bundle $F$.\\
\end{definition}

Now, we consider the entropy $\ss_j^\pm$ of the fiber bundles $E_0^\pm$ invariant by the flow $\phi^j_t$ of $X_j$ (the volume form is $dM_j$). As $E_0^+\oplus E_0^-$ admits a non degenerate bilinear antisymmetric form that is invariant by $\phi$, we obtain that the entropy of $E_0^+\oplus E_0^-$ is zero\footnote{Remember that a symplectic matrix $S$ has determinant $1$, and from $(d\alpha_j)_p(u,v) =(d\alpha_j)_{\phi_t(p)}(d\phi_t u,d\phi_t v)$
we obtain
$\big(\det (\phi_t)_{|_{E_0^+\oplus E_0^-}}\big)^2 = 1$
} , that is $$\ss_j^+ + \ss_j^-=0$$ 

We choose our maps $C_{j}^\pm$ such that
\begin{align}\label{constantscondition}
\frac{\delta_{ij}\ss_j^\pm}{p} = C_{j}^\pm(X_i)
\end{align}
where $p$ is the dimension of $E_0^\pm$.
\begin{remark}
Notice that hyperbolicity implies that $\ss_j$ is non zero.
\end{remark}

\begin{proof}[Proof of Lemma \ref{intbeta}]
In what follows we avoid using the index $j$ that indicates which connection we are using. But whenever the index $j$ is used, it indicates the very same index of the connection.
	
Let $\tau(t)$ be the parallel transport of $\nabla$ along the orbits of the flow of $X_j$ starting at the point $p_0$. Let $\phi_t$ be the flow along ${X_j}$, and let $p = \phi_t(p_0)$. From the definition of $\nabla$, a vector field $Y \in E_0^+$ is parallel with respect to $\phi$ if 
$$0 = \nabla_{\dot\phi_t}Y = \nabla_{{X_j}}Y = [{X_j},Y] + C_{j}^+(X_j)Y$$

If we denote $$s =C_{j}^+(X_j) = \frac{\ss_j^{+}}{p}$$
Then $Y$ is parallel with respect to $\phi$ if $[{X_j},Y] +sY=0$.  If we write $Y_p = \tau(t)Y_0$ then, $$\tau(t)_{|_{E_0^\pm}} = e^{-st}d\phi_t$$

In fact:
\begin{align*}
\L_{{X_j}}Y  &= \L_{X_j}(e^{-st}d\phi_tY_0)\\
&=\frac{d}{dt}(e^{-st})d\phi_tY_0 + e^{-st}\L_{X_j}(d\phi_tY_0)\\
\end{align*}

But $\L_{X_j}(d\phi_tY_0) =0$ because $(d\phi_tY_0)_{\phi_s} = d\phi_{t+s}Y_0$ and thus
\begin{align*}
\lim_{s\to 0}\frac{d\phi_{-s}d\phi_t(Y_0)_{\phi_s(p_0)} - d\phi_t(Y_0)}{s} = 
\lim_{s\to 0}\frac{d\phi_{-s}d\phi_{t+s}(Y_0) - d\phi_t(Y_0)}{s} = \lim_{s\to 0}\frac{d\phi_{t}(Y_0) - d\phi_t(Y_0)}{s} = 0
\end{align*}

On the other hand $\frac{d}{dt}(e^{-st}) = -se^{-st}$, and thus, $[{X_j},Y] = -sY$ and therefore $Y$ is parallel as we wanted.\\

Now, let $\Delta = \Delta^j(t)$ be the determinant of the parallel transport restricted to $E_0^+$ (the section $\zeta$ is used to compute the determinant). We have
$$\Delta(t) = e^{-pst}det(d\phi^{\upsilon}_t|_{E^{+}})$$

From $\nabla {X_j}=0$ we obtain
$$\beta({X_j}) \zeta({X_j}) =(\nabla_{{X_j}}\zeta)({X_j}) = {X_j}(\zeta({X_j}))$$

Also, as $\zeta$ was chosen to be never vanishing, we have:
\begin{align*}
\beta^{j}({X_j}) &= \frac{{X_j}(\zeta^{j}({X_j}))}{\zeta^{j}({X_j})} \\
&= \frac{d}{dt}\log (\Delta(t))\\
&=\frac{d}{dt}\big(\log (\det(d\phi_t)|_{E^{+}})\big) - ps\\
&=\frac{d}{dt}\big(\log (\det(d\phi_t)|_{E^{+}})\big) - \ss^{\pm}
\end{align*}

But, $\ss^+$ is defined by
\begin{align*}
\ss^{+} Vol(M) t = \int_{M}\log (\det(d\phi_t)|_{E^{+}})
dM_j
\end{align*}
Taking the derivative with respect to $t$ on both sides we obtain
\begin{align*}
\int_{M}\beta^{j}({X_j})dM_j=0
\end{align*}

\end{proof}

\begin{lemma}
Consider a topologically transitive action $\psi:G\times N\to N$ on a manifold $N$ and a $G$ invariant smooth splitting $TN:E\oplus F$. Let $\Theta$ be a, nowhere zero, smooth section of $\Lambda^{dim(F)}F^*$ and let  $\alpha,\beta$ be $\psi$-invariant smooth  sections of $\Lambda^{dim(E)}E^*$. Suppose that $\alpha\wedge\Theta$ is a volume form, then $$\beta\wedge\Theta = C\alpha\wedge\Theta\enskip for\; some\; constant\; C$$
\end{lemma}
\begin{proof}

Notice that as $\beta\wedge\Theta$ is a top form on $N$, it follows that $\beta\wedge\Theta = F\alpha\wedge\Theta$ for some function smooth $F$. We must show that this function is constant. We shall prove this, by showing that $F$ is constant over a dense subset.\\

    Because the splitting is $\psi$ invariant and $\Theta$ is a top form over $F$ we have $\psi_v^*\Theta = f_v\Theta$ for some function $f$ over $M$ (which depends on the parameter $v\in G$).
    Let us see that this function is nowhere zero. As $\Theta$ is a nowhere zero section of $\Lambda^{dim(F)}F^*$, it follows, that for any local frame $Y_1,\dots,Y_n$ of $F$, we have $\Theta(Y_1,\dots,Y_n)\neq 0$. Moreover, as $\psi_v$ is a diffeomorphism which preserves the splitting, it follows that if $Y_1,\dots,Y_n$ is a local fram of $F$, then $(\psi_v)_*Y_1,\dots,(\psi_v)_*Y_n$ is also a local frame of $F$ (in another neighborhood), and therefore
    $$0\neq\Theta((\psi_v)_*Y_1,\dots,(\psi_v)_*Y_n)=\psi_v^*\Theta(Y_1,\dots,Y_n) = f_v\Theta(Y_1,\dots,Y_n)$$
    thus, $f_v$ must be non zero.
    
    Let $x\in N$ such that $\{\psi_v(x)\}_{v\in G}$ is dense. We have:
    \begin{align*}
        (\psi_v^*\beta\wedge\Theta)_x&= \big(\psi_v^*\beta)\wedge(\psi_v^*\Theta)\big)_x\\
        &=
        \big(f_v\beta\wedge\Theta\big)_x = \big(f_vF\alpha\wedge\Theta\big)_x
    \end{align*}
    on the other hand
    \begin{align*}
         (\psi_v^*\beta\wedge\Theta)_v&=\big(\psi_v^*(F\alpha\wedge\Theta)\big)_x\\
         &=\big((\psi_v^*F)f_v\alpha\wedge\Theta\big)_x 
    \end{align*}
    we have therefore
    \begin{align*}
        f_v(x)F(x)(\alpha\wedge\Theta)_x = F(\psi_v(x))f_v(x)(\alpha\wedge\Theta)_x
    \end{align*}
    As $(\alpha\wedge\Theta)_x$ is non zero, it follows that 
    $$F(\psi_v(x)) = F(x)\enskip\enskip \forall\; v \in G$$
    Thus, $F$ is constant along the orbit of $x$ which is dense, as we desired.
\end{proof}

\begin{lemma}\label{obscuro2}
Let $\Theta_j$ the projection of $\Theta$ over $\Lambda^q(E_j)^*$, that is
$$\Theta_j(Y_1,\dots,Y_q):=\Theta(p_j(Y_1),\dots,p_j(Y_q))$$
	For every $s\in \{1,\dots,p\}$, we have $(\Omega^j)^{s}\wedge {d\alpha_j}^{p-s}\wedge\Theta_j=0$.
\end{lemma}
\begin{proof}
	We only prove in the case $j=1$. The other cases are similar. We will forego the index and write $\Theta=\Theta_j$ and $\Omega = \Omega^k$. In particular, from here on, $\Omega^s$ means $\Omega\wedge\dots\wedge\Omega$. 
	It is clear that the form $\Omega$ is invariant by the action, but by the topological transitivity of the action and the previous lemma, there exists constants $c_s$  such that 
	$$\alpha_1\wedge\dots\wedge\alpha_k\wedge\overbrace{\Omega^s}^{\Omega\wedge\dots\wedge\Omega}\wedge{d\alpha_1}^{p-s}\wedge\Theta = c_sdM_1 $$

If these constants $c_s$ are zero, the proof is finished. So we shall suppose that they are not zero and $\alpha_1\wedge\dots\wedge\alpha_k\wedge\Omega^s\wedge d\alpha_1^{p-s}\wedge\Theta$ is a true volume form on $M$.

Thus, using $\Omega = d\beta$ and integrating by parts, we have:
\begin{align*}
c_s\int dM_1 &= \int\alpha_1\wedge\dots\wedge\alpha_k\wedge\Omega^s\wedge{d\alpha_1}^{p-s}\wedge\Theta\\ &=\int\beta\wedge\alpha_2\wedge\dots\wedge\alpha_{k}\wedge \Omega^{s-1}\wedge{d\alpha_1}^{p-s+1}\wedge\Theta
\end{align*}

Now, we write $\beta = \sum_jf_j\alpha_j + \eta$ where $\eta$ vanishes on $T\phi$.
 Clearly\footnote{
We have seen that $K(X_s,\;\cdot\;) = 0$ for every $s\in \{1,\dots,l\}$, and thus, $i_{X_s}\Omega = 0$. Thus if we write  $\beta = \sum_jf_j\alpha_j + \eta$, then
\begin{align*}
\beta\wedge\alpha_2\wedge\dots\wedge\alpha_{k}\wedge &\Omega^{s-1}\wedge{d\alpha_1}^{p-s+1}\wedge\Theta =\\
&= f_1\alpha_1\wedge\dots\wedge\alpha_{k}\wedge \Omega^{s-1}\wedge{d\alpha_1}^{p-s+1}\wedge\Theta + 
\eta\wedge\alpha_2\wedge\dots\wedge\alpha_{k}\wedge \Omega^{s-1}\wedge{d\alpha_1}^{p-s+1}\wedge\Theta
\end{align*}
but $\eta\wedge\alpha_2\wedge\dots\wedge\alpha_{k}\wedge \Omega^{p-1}\wedge{d\alpha_k}^{n-p+1}$ is a volume form, and thus differs from $dM_1$ by a function $g$. As
$i_{X_1}\eta\alpha_2\wedge\dots\wedge\alpha_{k}\wedge \Omega^{p-1}\wedge{d\alpha_1}^{n-p+1} = 0$ it follows that this function is null.
}  
$$\beta\wedge\alpha_2\wedge\dots\wedge\alpha_{k}\wedge \Omega^{s-1}\wedge{d\alpha_1}^{p-s+1}\wedge\Theta = f_1\alpha_1\wedge\dots\wedge\alpha_{1}\wedge \Omega^{s-1}\wedge{d\alpha_1}^{p-s+1}\wedge\Theta$$
and from $\eta_{|_{T\phi}}=0$ it follows $\beta(X_1) = f_1$. And thus

\begin{align*}
c_s\int dM_1 &=\int\beta^1(X_1)\alpha_1\wedge\cdots\wedge\alpha_1\Omega^{s-1}\wedge{d\alpha_1}^{p-s+1}\wedge\Theta\\
&=c_{s-1}\int\beta(X^1)dM_1 =0\\
\end{align*}
as wanted.
\end{proof}

We are finally ready to prove Lemma \ref{nilpotentB}:

\begin{proof}[Proof of Lemma (\ref{nilpotentB})]
	First we notice that $B^j$ leaves $E_0^+$ invariant (this follows from Lema \ref{curv2}). Consider now the eigenvalues $\xi_1,\dots,\xi_p$ of $B^j|_{E_0^+}$.  It follows from the Lemma \ref{obscuro2} that, for every $s \in \{1,\dots,p\}$ we have
	\begin{align}\label{somaloca2}
	\sum_{i_1<\dots<i_s}\xi_{i_1}\dots\xi_{i_s}=0
	\end{align}
	And thus, $\xi_i=0$ for every $i$, and thus $B^j$ is nilpotent.  
\end{proof}

        \section{Explicit Levi decomposition and special closed subgroups}\label{moretecsec}
            As we mentioned before in Subsection \ref{subsecaffine}, due to Lemma \ref{closedimpliesclosed} to build our model space, we must show that the connected Lie subgroup $H\subset G$ corresponding with the Lie subalgebra $\lieh\subset\lieg$ is closed. As we have seen, the Lie algebra $\lieg$ is reductive, on this section we construct, in a more explicit way the semisimple part of the Levi decomposition of $\lieg$. 

\subsection{Partial connections}
We will need a couple of results from the theory of partial connections, a notion first introduced by R. Bott in \cite{Bott}. Intuitively, partial connections behave like usual connections, but we are only allowed to differentiate sections on some dirrections.

\begin{definition}
Given a vector bundle $\pi:E\to M$ and a involutive distribution $D\subset TM$, we define a partial connection along $D$ as an $\R$-bilinear map
$$\nabla^D:\Gamma(D)\times\Gamma(E)\to\Gamma(E)$$
satisfying, for every $Z\in \gamma (D)$, $s\in \Gamma(E)$ and $f\in C^{\infty}(M)$,
\begin{itemize}
    \item $\nabla_{fZ}s = f\nabla_Zs$ 
    \item $\nabla_Zfs = X(f)s + f\nabla_Zs$
\end{itemize}
\end{definition}

Due to the integrability condition on the distribution $D$, we can define the  curvature of a  partial connections in the same way we do for usual connections. The following Lemma show us that much like in the usual case, the curvature is the obstruction for existence of local parallel sections:
\begin{lemma}[Rawnslay, J. \cite{Rawnslay}]
A partial connection over $D\subset TM$ on a  $C^\infty$-vector bundle $\pi:E\to M$ has zero curvature, if and only if, $E$ can be trivialized (locally) by sections $s_1,\dots,s_{rank(E)}$ satisfying $\nabla^Ds_i = 0$
\end{lemma}

Moreover, we shall see that, as usual, the topology of $M$ is the obstruction for the construction of global sections.\\

In the case of usual connections, if we want to extend a local parallel section to a global one, we do the following construction.

Let $s$ be a local parallel section defined on $U\subset M$. For any point $x\in M$ choose a path from $U$ to $x$, and extend $s$ to $x$ via the parallel transport. What happens if we take another path? 

For an usual connection, the failure of this construction is measured by the holonomy, in particular, if $M$ is simply connected, then there is no holonomy and the extension of $s$ above is in fact well defined. \\

For partial connections, we need a little more care, but in the end, simply connectedness will allow a similar construction:\\

Let $U\subset M$ be an open subset on which $E|_U$ admits a trivialization by flat sections. Denote by $\overline U$ the saturation of $U$ by leaves of $D$ and suppose that $M = \overline U$. Then, every point $x$ in $M$ can be joined bay a point in $U$ such that the path is tangent to $D$. We can, therefore, try to extend the local sections to $x$ by the parallel transport along this path. 

At first glance, simply connectedness doesn't seems to be enough to ensure that the extension is well defined, because the only admissible paths are paths tangent to $D$. The following Theorem show us that this is not a problem:

\begin{theorem}[A. Lerario, A. Mondino, \cite{Lerario}]
Let us consider the space $\mathcal L^D$ of loops tangent to the distribution $D$. Then $\mathcal L^D$ has the
homotopy type of a CW-complex and for all $k \geq 0$:
$$ \pi_k(\mathcal L^D) = \pi_k(M) \ltimes \pi_{k+1}(M)$$
\end{theorem}
\begin{coro}
If $M$ is simply connected, then $\mathcal L^D$ is path connected
\end{coro}
The above Corollary means that if $M$ is simply connected, then every loop tangent to $D$ can be deformed to the constant map in such a way that the loops in between are also tangent to $D$.

\subsection{About the determinant bundle}
We are now, going to apply the results from the previous subsection to the bundle $\Lambda_0^\pm$ with the induced connections $\nabla^j$. 

As the distribution $T\phi$ is involutive, we can restrict the connection $\nabla^j$ and obtain a partial connection over $T\phi$. Let us show that this partial connection is flat:

\begin{lemma}\label{cur}
Let $K = K^j = (K^j)^\pm$ be the curvature of the connection $\nabla = \nabla^j$ over the bundle $\Lambda_0^\pm$. Then, for Y=$Y_j^\pm$ and $Z_j^\pm$ tangent to $E_j^\pm$,
\begin{enumerate}
    \item $K(X_s,X_l)= 0$
    \item $K(X_s, Y_1^\pm) = 0$
    \item $K(Y_j^-,Z_j^-) = 0$
    \item $K(Y_j^+,Z_j^+) = 0$
\end{enumerate}
\end{lemma}
\begin{proof}
This is an easy consequence of Lemma \ref{curvlemma}.
\end{proof}

Using this lemma and the results of the previous subsection, we obtain the following:
\begin{lemma}
Consider the lift $\tilde\Lambda_0^\pm$ of $\Lambda_0^\pm$ to the universal cover $\tilde M$ of $M$, with associated connections $\nabla^j$. Then, there exists unique, up to a constant, non zero, sections $\zeta_j$ which are parallel (with respect to $\nabla_j$) along $\tilde{T\phi}\oplus \tilde E_j^+\oplus\tilde E_j^-$ 
\end{lemma}
\begin{proof}
First we construct sections $\zeta_j$ which are parallel along $\tilde{T\phi}$.
The existence of such sections is obtained by observing that $\tilde M$ is simply connected, the action is topologicaly transitive, and therefore, for any open set $U\subset M$ the saturation of $U$ by the leafs of $T\phi$ is in fact $M$. This of course lifts to the universal cover. Finally, the induced partial connections $\nabla^j$ of $\tilde\Lambda_0^\pm$ are flat. The uniqueness, up to a constant follows, once again,  from the topologically transitiveness of the action.

Now, the sections $\zeta_j$ are only parallel along $\tilde{T\phi}$. We want to show that they are also parallel along $\tilde E_j= \tilde E_j^+\oplus \tilde E_j^-$. We repeat the construction above for the integrable distributions $T\phi\oplus E_j^\pm$, obtaining unique (up to constants) sections $\zeta_j^\pm$ which are parallel along  $\tilde{T\phi}\oplus \tilde E_j^\pm$. However, $\zeta_j^\pm$ are also parallel along $\tilde{T\phi}$ and thus, up to a constant, $\zeta_j^\pm = \zeta_j$. That is, $\zeta_j$ is parallel along $\tilde{T\phi}\oplus \tilde E_j$ as we desired.\footnote{We can't just do the construction directly with the bundle $T\phi\oplus E_j$ because Lemma \ref{cur} does not says what happens for $K(Y_j^+,Z_j^-)$, and thus we can not say that the curvature vanishes on $T\phi\oplus E_j^+\oplus E_1^-$, Lemma \ref{cur} only ensure us that the curvature, restricted to  $T\phi\oplus E_j^\pm$ vanishes}

\end{proof}

\begin{definition}\label{characters}
	We define the maps $d\chi_j:\lieg'\to\R$ by:
	$$\forall Y\in\lieg'\;\;\L_Y(\zeta_j) = d\chi_j(Y)\zeta_j$$
	The uniqueness of $\zeta_j$ means that this is a well defined map
\end{definition}

\begin{lemma}
For $X\in \Ii':= Span_\R{X_1,\dots,X_k}$ and $\zeta$ a section of $\Lambda^+$ we have
	$$\nabla^j_X\zeta - \nabla^i_X\zeta = p(C_{i}^\pm(X) -C_{j}^\pm(X))\zeta$$
\end{lemma} 
\begin{proof}
	
	First, for any $Y = Y_0^\pm$ tangent to $E_0^\pm$ we have:
	$$\nabla^j_XY- \nabla^i_XY = (C_{j}^\pm(X) -C_{i}^\pm(X))Y$$
	
	Now we take $Y_1,\dots,Y_p$ local basis for $E_0^\pm$,  from the formula
	$$\nabla_X\zeta(Y_1,\dots,Y_p) = X(\zeta(Y_1,\dots,Y_p)) - \sum_l\zeta(Y_1,\dots,\nabla_XY_l,\dots,Y_p)$$
	we obtain
	\begin{align*}
	(\nabla^j_X\zeta - \nabla^i_X\zeta)(Y_1,\dots,Y_p) &= \sum_l\zeta(Y_1,\dots,\nabla^i_XY_l - \nabla^j_XY_l ,\dots,Y_p)\\
	&= \sum_l(C_{i}^\pm(X) -C_{j}^\pm(X)) \zeta(Y_1,\dots,Y_l ,\dots,Y_p)\\
	&=  p(C_{i}^\pm(X) -C_{j}^\pm(X))\zeta(Y_1,\dots,Y_l ,\dots,Y_p)
	\end{align*}
\end{proof}

\begin{lemma}
	Let us write $\zeta_i = f_{ij}\zeta_j$ then  $X(f_{ij})= 0$ if, and only if $(C_{j}^\pm(X) -C_{i}^\pm(X))=0$. 
\end{lemma}
\begin{proof}
	From the previous lemma, 
	\begin{align*}
	p(C_{i}^\pm(X) -C_{j}^\pm(X))\zeta_i &= \nabla^j_X\zeta_i - \nabla^i_X\zeta_i=\nabla^j_X\zeta_i  \\
	&= \nabla^j_Xf_{ij}\zeta_j\\
	&=	X(f_{ij})\zeta_j + f_{ij}\nabla^j_{X}\zeta_j\\
	&=\frac{X(f_{ij})}{f_{ij}}\zeta_i
	\end{align*}
	that is $X(f_{ij}) = 0$ if and only if 
	$$(C_{j}^\pm(X) -C_{i}^\pm(X))=0$$. 

\end{proof}

\begin{remark}
	From $C_j^\pm(X_s) = \frac{\delta_{js}\ss^\pm}{p}$ and the previous lemma, we obtained that, for $i\neq j$,
	$$\frac{X_i(f_{ij})}{f_{ij}} = \ss_i\enskip\enskip and \enskip\enskip \frac{X_j(f_{ij})}{f_{ij}} = -\ss_j$$
	\end{remark}

\begin{lemma}\label{l}
	For  every $1\leq j\leq l$, let $d\chi_j:\lieg' \to \R$ be the character\footnote{Once again we avoid the usage of the superscript $"\pm"$. The more correct notation is $d\chi_j^\pm$}, given by Definition \ref{characters}. We have:
	$$\L_{X_j}\zeta_j = \ss_j\zeta_j$$
	Moreover $d\chi_j(L_0)  = Tr_{|_{E^\pm}}(L_0)>0$.
\end{lemma}
\begin{proof}
	First of all, notice that $d\chi_j$ is in fact a Lie algebra morphism. In fact, 
	\begin{align*}
	d\chi_j([Y,Z])\zeta_j &= \L_{[Y,Z]}\zeta_j = [\L_Y,\L_Z]\zeta_j\\
	&=\L_Y\L_Z\zeta_j - \L_Z\L_Y\zeta_j\\
	&=\L_Y(d\chi_j(Z)\zeta_j) - \L_Z(d\chi_j(Y)\zeta_j)\\
	&=d\chi_j(Y)d\chi_j(Z)\zeta_j - d\chi_j(Z) d\chi_j(Y)\zeta_j = [d\chi_j(Y),d\chi_j(Z)]\zeta_j
	\end{align*}
	As wanted. Now, let's prove $d\chi_j(X_j) = \ss_j$.
	
	We shall write 
	$$Y = (Y_0,\dots,Y_p),$$
	$$Y^{(s)} = (Y_0,\dots,Y_{s-1},Y_{s+1},\dots,Y_p)$$ and for indexes $\alpha_1<\dots<\alpha_l$, $l<p$
	$$Y^{(\alpha_1\cdots\alpha_l)} = (Y_0,\dots,Y_{\alpha_1-1},Y_{\alpha_1+1},\dots,Y_{\alpha_l-1},Y_{\alpha_l+1},\dots,Y_p)$$
	Let $\nabla$ be a connection on $\Lambda_0^\pm$ and $\zeta$ be a parallel section along $Y_0$. 
	\begin{align*}
	0 =\nabla_{Y_0}\zeta(Y^{(0)})&=Y_0(\zeta(Y^{(0)})) - \sum_s \zeta(Y_1,\dots,\nabla_{Y_0}Y_s,\dots,Y_p)\\
	d\zeta(Y) &=\sum_{s=0}^p(-1)^sY_s(\zeta(Y^{(s)})) + \sum_{i<s}(-1)^{i+s}\zeta([Y_i,Y_s],Y^{is}) \\
	&=\sum_{s=0}^p(-1)^sY_s(\zeta(Y^{(s)})) + \sum_{0<i<s}(-1)^{i+s}\zeta([Y_i,Y_s],Y^{is}) + \sum_{s=1}^p(-1)^s\zeta([Y_0,Y_s],Y^{0s})\\
	\end{align*}
	Also
\begin{align*}
d(i_{Y_0}\zeta)(Y^{(0)}) &= \sum_{s=1}^p(-1)^{s-1}Y_s(\zeta(Y_0,Y^{(0s)})) + \sum_{0<i<s}(-1)^{i+s}\zeta(Y_0,[Y_i,Y_s],Y^{(0is)})\\
&= Y_0(\zeta(Y))-\sum_{s=0}^p(-1)^{s}Y_s(\zeta(Y^{(s)})) - \sum_{0<i<s}(-1)^{i+s}\zeta([Y_i,Y_s],Y^{(is)})\\
&=Y_0(\zeta(Y)) -d\zeta(Y) + \sum_{s=1}^p(-1)^s\zeta([Y_0,Y_s],Y^{0s}) \\
&=\sum_{s=1}^p\zeta(Y_1,\dots,\nabla_{Y_0}Y_s -[Y_0,Y_s],\dots,Y_p) -d\zeta(Y)
\end{align*}
Now, we notice that for $Y_0 = X_j$,  $\nabla = \nabla^j$ and $\zeta = \zeta_j$ we have $\nabla^j_{X_j}Y = [X_j,Y] + C_j^\pm(X_j) Y$, thus, if we take $Y_1,\dots,Y_p$ a local basis of sections of $E_0^\pm$ we obtain $\nabla^j_{X_j}Y_l - [X_j,Y_l] = C_{j}^\pm(X_j) Y_l$ for every $l=1,\dots,p$, that is:
\begin{align*}
    \L_{X_j}\zeta(Y^{(0)}) &= d(i_{X_j}\zeta)(Y^{(0)}) + d\zeta(X_j,Y^{(0)})\\
    &=\sum_{k=1}^n\zeta(Y_1,\dots,C_j^\pm(X_j) Y_k,\dots,Y_n) = pC_{j}^\pm(X_j)\zeta(Y^{(0)}) 
\end{align*}
but $pC_{j}^\pm(X_j) = \ss_j(=\ss_j^\pm)$ and thus $\L_{X_j}\zeta_j = \ss_j\zeta_j$ as we wanted.\\

Now, for any $Y\in \lieh'$ we have $d\chi(Y) = Tr_{|_{E_0^\pm}}(Y)$. Here we understand $Y$ as an element on $End(V_0)$. As this immersion actually depends on the connection\footnote{Remember, the map $\theta_0:\lieg'\to End(V_0)\times V_0$ is defined as $\theta_0(Y):=(\L_Y - \nabla_Y,Y_{v_0})$}, the trace may depend on the choice of Anosov element. However, as we have seen, when restricted to $\lieh'$ the map $\theta_0$ is actually de differential of the map $j:H'\to GL(V_0)$, $j(h) = dh_{v_0}$ which does not depend on the choice of the connection. Thus, for $Y\in \lieh'$ we have 
\begin{align}\label{tradedxi}
d\chi(Y) = Tr_{|_{E_0^\pm}}(Y)
\end{align}
 and the trace formula doesn't depends on the choice of Anosov element and associated connection, etc.
\end{proof}
\begin{remark}
We recall that the sequence
$$0\longrightarrow \liek_j'\longrightarrow\lieg'\stackrel{\Check P_{j}}{\longrightarrow}\lieg_j\longrightarrow 0$$
is split. Let $\sigma_j:\lieg_j\to\lieg'$ be the corresponding embedding. 
    If we take the projection $L_j = \sigma_j(\Check P_{j}(L_0))$, then it is clear that $L_j$ also belongs to $\lieh'$, moreover, as $\liek'_j$ is an ideal, the action of $L_j$ on $E_0$ coincides with the action of $L_0$, and thus, it has positive eigenvalues and we obtain
    $$d\chi(L_j) = Tr_{|_{E_0^\pm}}(L_j) = Tr_{|_{E_0^\pm}}(L_0)>0$$
\end{remark}

\begin{remark}
Consider the Lie algebra $\sigma_j(\lieg_j)\subset\lieg'$ and the intersections $\ker(d\chi_i)\cap\sigma_j(\lieg_j)$, for any $i,j$. As $\ker(d\chi_i)$ has codimension 1, this intersection must have codimension  at most one. 

Now, from Remark \ref{remsplit},  $X_i$ actually belongs to $\lieg_j$ for every $1\leq i,j\leq l$, and thus the codimension must be $1$.
\end{remark}

In the following, we shall actually identify $\lieg_s$ with it's image $\sigma_s(\lieg_s)\subset\lieg'$.

\begin{coro}\label{cocoro}
For any $s\in \{1,\dots,l\}$, we have:
	$$X_i - \frac{\ss_iL_s}{d\chi_i(L_s)}\in \lieg_s\cap \ker d\chi_i \cap (\ker d\chi_{j})^c\enskip\enskip i\neq j$$
\end{coro}
\begin{proof}
    That $X_i - \frac{\ss_iL_s}{d\chi_i(L_s)}\in \lieg_s$ it is clear.\\
    
    Now, we recall that $d\chi_i$ restricted to $\lieh'$ is given by the trace, which doesn't actually depend on $i$, and thus, $d\chi_i(L_s) = d\chi_j(L_s)$ and write $r_i = \frac{\ss_i}{d\chi_i(L_s)}$. We also write $f = f_{ij}$ (that is: $\zeta_i = f\zeta_j$).\\
		
	That $X_i - r_iL_s\in \ker d\chi_i $ is also clear. Now notice that
	\begin{align*}
	\L_{{X_j} - r_jL_s}\zeta_i &= 
	\L_{{X_j}}f\zeta_j - r_jd\chi_i(L_s)\zeta_i\\
	&=
	{X_j}(f)\zeta_j + f\L_{{X_j}}\zeta_j -r_jd\chi_i(L_s)\zeta_i\\
	&=\big(\frac{{X_j}(f)}{f} + \ss_j - \frac{\ss_jd\chi_i(L_s)}{d\chi_j(L_s)}\big)\zeta_i = -\ss_j\zeta_i\neq 0.
	\end{align*}
\end{proof}

\begin{coro}\label{coro2}
	For any $i, j\in \{1,\dots,l\}$, $i\neq j$ we have $d\chi_j(X_i) =0$.
\end{coro}
\begin{proof}
	 From previous lemmas, we have
	\begin{align*}
	\L_{{X_j}}\zeta_i&=\L_{{X_j}}(f_{ij}\zeta_j) = 
	{X_j}(f_{ij})\zeta_j + f_{ij}\L_{{X_j}}\zeta_j \\
&=\bigg(\frac{{X_j}(f_{ij})}{f_{ij}} + \ss_j\bigg)\zeta_i = 0
	\end{align*}
\end{proof}

\begin{remark}\label{rem}
	Notice that the Corollary \ref{cocoro} implies that, for any $s$, the hyperplanes $\ker d\chi_j\cap \lieg_s$, $j=1,\dots,l$, are in general position, that is $\cap_j \ker d\chi_j$ have co-dimension $l$ on $\lieg_s$. To see this, we notice that  for every $j$, 
	$$\bigg(X_j - \frac{\ss_jL_s}{d\chi_j(L_s)}\bigg)\not\in\lieg_s\cap\bigcap_j \ker d\chi_j$$
	It is clear that $\{X_1 - \frac{\ss_1L_s}{d\chi_1(L_s)},\dots,X_l - \frac{\ss_lL_s}{d\chi_l(L_s)}\}$ are linearly independent.
\end{remark}

\begin{remark}
	Consider the Lie algebra 
	$$\lieg_s^0:=\bigcap_{i=1}^l\big(\ker(d\chi_i)\cap\lieg_s\big)$$
	
	Then, the map $\pi:\lieg_s\to \lieg_s\slash \Ii_s$ restricted to $\lieg_s^0$ is an isomorphism. In fact, from the Remark \ref{rem}, this map is indeed between spaces of the same dimension, thus, we just have to prove that it is an injection. Now, let $X = \sum_j b_jX_j$ be an element of $\Ii_s = \ker(\pi)$. Then
	$$ d\chi_i(X) = \sum_jb_jd\chi_i(X_j) = b_i\ss_i$$

	That is, $X\in \lieg_s^0  \Leftrightarrow b_i=0\forall i$, that is the map is injective.

\end{remark}
\begin{coro}\label{cccoc}
	The Lie algebra $\lieg_s^0$ is semisimple
\end{coro}
\begin{proof}
We recall that $\lieg_s$ is reductive and its center is $\Ii_s$. From the previous remark $\lieg_s^0 = \lieg_s\slash\Ii_s$ and thus is semisimple.
\end{proof}

We can also build a Lie group $H^0\subset H'$, by taking characters $\chi_j:H'\to (\R,\;\cdot\;)$ defined by
$$h_*\zeta_j = \chi_j(h)\zeta_j$$
and taking $H^0 = \cap_j\ker\chi_j$.

We have the split sequence
$$0\to\liek'\to\lieg'\stackrel{\sigma}{\leftrightarrows}\lieg\to 0$$
thus,  we can identify, $\lieg^0$ as a subalgebra $\sigma(\lieg^0)\subset \lieg'$. It is clear that $Lie(H^0)\cap \sigma(\lieg^0) = \sigma(\lieg^0\cap \lieh) = \sigma(\lieh^0)$. Moreover, it is clear that $H^0$ is a closed subgroup of $H'$.

            We are ready to construct our local model. For any fixed $j\in \{1,\dots,l\}$, we consider the following data
\begin{itemize}
    \item Lie algebras $\lieg'$ and $\lieg = \lieg_j$ with corresponding connected, simply connected subgroups $\hat G'$ and $\hat G=\hat G_j$.
    \item Lie subalgebras $\lieh'\subset\lieg'$ and $\lieh=\lieh_j,\lieg^0 =\lieg_j^0,\lieh^0 = \lieh_j^0\subset \lieg_j$ with corresponding connected subgroups $\hat H'$, $\hat H = \hat H_j$, $\hat G^0 =\hat G_j^0$ and $\hat H^0 = \hat H_j^0$.
\end{itemize}
As previously mentioned, our goal is to model our space as the quotient $\hat G'\slash \hat H'$. The map $\Check P = \Check P_j:\lieg'\to\lieg_j$ induces a map $P:\hat G'\to\hat G$ such that $P(H') = H$. It follows from Lemma \ref{closedimpliesclosed} that it suffices to show that $\hat H$ is closed in $\hat G$. We will make a further reduction:

\begin{remark}
	Notice that if $\hat G^{00}$ is the connected, simply connected, Lie group associated with $\lieg^0$, then $\hat G^{00}\times \R^k$ is a connected, simply connected, Lie group with Lie algebra $\lieg^0\oplus \R^k = \lieg$, and thus, $\hat G = \hat G^{00}\times \R^k$. Thus we can suppose that $\hat G^0$ is also simply connected.
	\end{remark}

\begin{lemma}
	Consider $G = G^0\times \R^k$ a Lie group and $H\subset G$ be a Lie subgroup. Then, $H$ is closed in $G$ if, and only if, $H\cap G^0$ is closed in $G^0$.
\end{lemma}	
\begin{proof}
	Notice that we only need to prove the lemma for the case $k=1$, the general case follows easily by induction.\\

	Consider $\lieg = \lieg^0\oplus\R$ the Lie algebra of $G$, and $\lieh$ the Lie algebra of $H$. If $\lieh\subset\lieg^0$ then $H\subset G^0$ and there is nothing to prove. Now, suppose that $\lieh\not\subset\lieg^0$, then, there exists an element $v\in\lieh$ with non zero $\R$ component. Take $R = \exp(\R v)$. The non zero $\R$ component of $v$ implies that $R$ is closed in $H$ and in $G$. We have the inclusion $H\slash R\subset G\slash R$, and the identifications
	$$H\slash R = G^0\cap H\enskip\enskip and\enskip\enskip G\slash R = G^0$$
	
	As the quotient map $\pi:G\to G\slash R$ is continuous, the inverse image of closed sets are closed, and thus, $\pi^{-1}(G^0\cap H) = H$ is closed.
	
\end{proof}	
\begin{lemma}
The subgroup $\hat H'\subset\hat G'$ is closed.
\end{lemma}
\begin{proof}
	First, lets recall that $H'$ is an algebraic subgroup of $Aut(V_0)$. We also have that $\lieg'\slash \lieh' =  V_0$ therefore the map  
	\begin{align*}
	j:H' &\rightarrow Aut(\lieg')\\
	h' &\mapsto (Y \mapsto h'_*(Y)),
	\end{align*}
	is a (finite) covering onto it's image \footnote{In fact, consider $Y = (A,Y_0)\in\lieg'\subset End(V_0)\times V_0$. Then, $h'_*Y = (dh'_{v_0}\circ A,dh'_{v_0}Y_0)$. But $\lieg'\slash \lieh' = V_0$ means that the the composite map $\lieg' \to End(V_0)\times V_0 \to V_0$ is surjective, and thus $h'_*Y = Y$ for all $Y\in \lieg'$ implies $dh'_{v_0}Y_0 = Y_0$ for all $Y_0\in V_0$ and thus $dh'_{v_0} = Id$. But we have already seen (Lemma \ref{injectiv}) that the map $H'\ni h'\to dh'_{v_0}\in GL(V_0)$ is a covering and thus, has discrete kernel $\Gamma$, and thus, $h\in \Gamma'$.}, moreover it is algebraic. 
	
	Thus, the images of $H^0$ and  $H'$ are both algebraic subgroups of $Aut(\lieg')$ and therefore closed. Also, as $j$ is a covering, it follows that $Lie(H') = Lie(j(H'))$ and $Lie(H^0) = Lie(j(H^0))$.
	

	As $\lieg^0$ is semisimple, the map $Ad:\hat G^0\to Aut(\lieg')$ is a covering onto it's image. This implies that $Lie(\hat G^0) = Lie (Ad(\hat G^0))$. From $\lieh^0\subset Lie(H^0)$ it follows that $Ad(\hat H^0)\subset j(H^0)$. 
	
	Now, the adjoint representation is continuous and thus, $Ad^{-1}(j(H^0))$ is closed. It remains to show that $Ad^{-1}(j(H^0)) = \hat H^0$, but this follows fom the fact that the differential of the map $j$ is actually the adjoint representation of $\lieh'$ on $\lieg'$. 
\end{proof}

        \section{Ehresmann's geometric structure}\label{modelsec}
            \begin{remark}\label{anoact}
As $\hat H'$ is closed we can take the homogeneous space $\hat V = \hat G' \slash \hat H'$ we consider the base point $\hat v_0 \in \hat V $ the induced flows on $\hat V$
given by 
\begin{equation*}
\hat\varphi^j_t (\hat v)=\exp(tX_j)\cdot \hat v
\end{equation*}

Clearly those flows commute and they define an action of $\R^k$ on $\hat V$. We shall call it the Anosov action on $\hat V$.
\end{remark}

\begin{remark}\label{foliations}
We consider the associated vector fields $\hat X_j$. We also consider $\hat E_i^\pm $ the tangent distribution to the foliation $\hat{\mathcal F}_i^\pm: =\{ g \hat Q_i^\pm\slash \hat H' | g \hat Q_i^\pm \in \hat G' \slash \hat Q_i^\pm \}$. where  $\hat Q_i^\pm$ are subgroups of $\hat G'$ associated with the subalgebras:
$$\lieq_i^\pm:=\{Y\in\lieg'\; Y_{v_0} \in (E_i^\pm)_{v_0}\}$$\\

It is clear that $T(\hat V) = T\hat\phi\oplus \hat E_0^+\oplus \hat E_0^-\oplus \hat E_1^+\oplus \hat E_1^-$.\\

Also, for $Y,Z\in \lieg'$ the alternating bilinear maps $$B_j(Y,Z):=(d\alpha_j)_{v_0}(Y_{v_0},Z_{v_0})$$
is well defined and $\hat H'$-invariant, which allow us to extends them to $2-$ forms $\omega_j$ on $\hat V$.
\end{remark}

The map

\begin{equation*}
\theta(exp(Y) \cdot v_0) = exp(Y) \cdot \hat v_0 
\end{equation*}
is a local diffeomorphism between a neighbourhood of $v_0$ and a neighbourhood of $\hat v_0$ such that $\theta_* X_j = \hat X_j $, $\theta_* E_i^\pm = \hat E_i^\pm$ and $\theta^*\omega_i = d\alpha_i$. We also use $\theta$ to define a connection $\hat\nabla$ on the bundle $T\hat\phi^1\oplus\hat E_1^+\oplus \hat E_1^-$, such that $\theta^*\hat{\tilde\nabla} = \tilde\nabla$, where $\tilde \nabla$ is the connection required for rigidity of our geometric structure.

\begin{remark}
        The local connections $\hat {\tilde{\nabla^j}}$ on $\hat {T\phi}^j\oplus\hat E_j^+\oplus \hat E_j^-$ over a neighborhood of $\hat v_0$ is $\hat G'$ invariant and extends to a connection $\hat{\tilde\nabla^j}$ on $\hat{T\phi}^j\oplus\hat E_j$ over $\hat G'\slash \hat H'$.
    
    In fact, we recall that the connections $\tilde\nabla^j$ can be given by a vector valued bilinear map (given by the Christofell symbols) which transforms in a certain specific way. The map $\theta$ copies over this christofell symbols and ensure that it is transformed in the right way.  The fact that $\lieg'\subset\tilde\lieg'$ means precisely that  this bilinear maps are invariant by the left action of $\hat G'$. As $H'$ is the isotropy group of the structure, this means that the such maps actually defines a left invariant connection on $\hat G'\slash \hat H'$. In particular, this means that the connection $\hat{\tilde\nabla}^j$ is given by a single bilinear $ad(\hat H')$-invariant map $$B^j:(\lieg'\slash\lieh')\times(\mathfrak e^j\slash\lieh')\to\mathfrak e^j\slash\lieh'$$ where $\mathfrak e^j\subset \lieg'$ is the Lie subalgebra of vector fields tangent to $T\phi^j\oplus E_j^+\oplus E_j^-$.
\end{remark}

The transitivity of $\Aut^{loc}(\sigma)$ on $\Omega$ gives us local charts from $\Omega$ to $\hat G'\slash \hat H'$  which preserves the structure.

\begin{remark}
	The local charts we defined above do not yet define a $(\hat G',\hat V)$-structure on $\Omega$, because we do not have any control of the transition maps, that is, we obtained an atlas 
	\begin{align}\label{atlas}
	\{f_i:U_i\to V_i\subset\hat V\}
	\end{align}
	 of $\Omega$, and the change of coordinates of this atlas
	$$f_i\circ f_j^{-1}:f_j(U_i\cap U_j)\to f_i(U_i\cap U_j)$$
	preserves the structure. However, we do not know if the transition maps $f_i\circ f_j^{-1}$ are the restriction of an element of $\hat G'$.\\ 
	
	In what follows, we shall construct a (slightly) larger group $G'$, such that the atlas \ref{atlas} will in fact be an atlas of a $(G',\hat V)$-structure on $\Omega$.
\end{remark}

Define the group $G'$ of diffeomorphisms of $\hat V$ that preserves $\hat E_i^\pm$,  $\hat X_j$, $\omega_j$ for every $j$ and the connection $\hat\nabla$. This is a Lie group and its identity connected component is precisely the image of $\hat G'$ in $Diff(\hat V)$. Define $H'$ the isotropy group of $\hat v_0$ in $G'$. The following lemma proves that there is no confusion\footnote{Remember that we have previously defined $H'$ as the isotropy group of  our chosen point $v_0\in M$, that is $H' = \Aut^{loc}_{v_0}(\sigma)$. As the map $\theta$ above is a local diffeomorphism from a neighbourhood of $v_0$ to a neighbourhood of $\hat v_0$ that preserves the structure, we identify $\Aut^{loc}(v_0) = \Aut^{loc}(\hat v_0)$}.

\begin{lemma}
	Up to coverings, we have the following identifications 
	\begin{align*}
	H' \equiv \Aut^{loc}(\hat v_0) \equiv \{ A \in \Aut&(\lieg') | A(\lieq_i^\pm) \subset \lieq_i^\pm; A(\lieh') \subset \lieh'; A(X_j) = X_j\\
	&\omega_j\circ (A\times A) = \omega_j\;\;;\; A\circ B\circ(A\times A) = B \}
	\end{align*}
	where $B$ denotes the bilinear map which defines the connection.
\end{lemma}

\begin{proof}
	First we observe that $H' \subset \Aut^{loc}(\hat v_0)$. This is clear, as global diffeomorphisms	 can be restricted to local ones.\\
	 
	Lets call the right side of the above equivalence $H_1$.
	Remember that $\Aut^{loc}(\hat v_0)$ is the group of germs at $\hat v_0$ of local diffeomorphisms of $\hat V$ that preserves $\hat X_j$, $\hat E_i^\pm$, $d\alpha_j$ and the connection $\tilde\nabla$, and thus, we can identify
	
	$$\Aut^{loc}(\hat v_0) = \Aut^{loc}_{v_0}(\sigma)$$
	As we have already seen, the map
	\begin{align*}
	j:\Aut^{loc}_{v_0}(\sigma) &\rightarrow Aut(\lieg')\\
	h' &\mapsto (Y \mapsto h'_*(Y)),
	\end{align*}
	is a finite covering onto it's image $H_1'$, in fact, $\Aut^{loc}_{v_0}(\sigma) = H_1'\ltimes F$, where $F$ is a finite group. Thus, the second equivalence is in fact, the injection of $H_1'$ onto $H_1$. 
	
	Now, as $\hat G'$ is simply connected, $$\Aut(\lieg') = \Aut(\hat G')$$ and thus, an element $A$ of $H_1$ is also an automorphism of $\hat G'$. The condition $A(\lieh') \subset \lieh'$ implies that it preserves $\hat H'$, and thus induces an diffeomorphism of $\hat V$. The other conditions on $A$ means that it preserves the $A$-structure, which shows that $H_1 \subset H'$.
	
	Thus we obtained the following:
	$$H_1\hookrightarrow H' \hookrightarrow H_1'\ltimes F \to H_1' \hookrightarrow H_1$$
	Where the composition is actually the identity.
\end{proof}

\begin{remark}\label{gext}
    This lemma means that there exists a finite group $F$ of local diffeomorphisms fixing $\hat v_0$, such that every local diffeomorphism of $\hat V$ that fixes $\hat v_0$ and preserves the structure is, up $F$ the restriction of an element of $H'$. 
    
    Suppose, now, that there exists a representation $\rho':G'\to Aut(F)$ such that the restriction $\rho'|_{H_1'}$ induces the semidirect product $H_1'\ltimes F$. Then we consider $G'' = G'\ltimes F'$ and $H'' = H'\ltimes F$. It is clear that $G''\slash H'' = G'\slash H' = \hat V$ and, thus, every local diffeomorphisms of $\hat V$ that preserves the structure is the restriction of an element of $G''$.
    
    In the following we shall show that this is indeed the case.
\end{remark}

\begin{lemma}
    There representation $\rho:H_1'\to Aut(F)$ extends to a representation of $G'$
\end{lemma}
\begin{proof}
Let $\rho$ denote the representation $H_1'\to Aut(F)$ which induces the semidirect product $Aut^{loc}_{v_0} = H_1'\ltimes F $. As $H_1'\hookrightarrow H' \hookrightarrow H_1'\ltimes F $, it follows that we can extend the representation $\rho$ to $H'$ in a trivial way:
$$H'\ni(g,\gamma)\mapsto \rho(g)$$

Now, we observe that the exact homotopy sequence give us
$$\cdots \to \pi_1(\hat V)\to \pi_0(H')\to\pi_0(G')\to \pi_0(\hat V)$$
As the first and last terms are trivial, it follows that $\pi_0(G')=\pi_0(H')$.
Finally, observe that, for any Lie group $G$ a representation on a finite set reduces to a representation of $\pi_0(G)$ and thus, the representation $\rho:H'\to Aut(F)$ actually extends to a representation of $\rho$.
\end{proof}

\begin{coro}
	Every diffeomorphism from a connected open set of $\hat V$ to another that preserves $\hat X_j$, $\hat E_i^\pm$, $\omega_j$ and $B$  is a restriction of an element of $G'$ 
\end{coro}

\begin{proof}
	It suffices to prove in the case of a diffeomorphism that fixes $\hat v_0$. The corollary follows from the equality $H'= \Aut^{loc}(\hat v_0)$.
\end{proof}

\begin{remark}\label{remarkk}
	Notice that we can identify $\hat V = \hat G'\slash\hat H' = G'\slash H'$. As $H'$ is algebraic, it follows that $H'$ and $G'$ have a finite number of connected components. This means that up to finite covering of $M$, we can suppose that $G' = \hat G'$.
\end{remark}
\begin{lemma}\label{model}
	The diffeomorphisms $\theta$ from an open set of $\Omega$ to an open set of $\hat V$ which preserves the structure forms a maximal atlas of a $(G',\hat V)$- structure over $\Omega$.
\end{lemma}
\begin{proof}
	By construction the domains of definition of those diffeomorphisms covers $\Omega$, and from the previous lemma, it follows that the change of local coordinates are given by the elements of $G'$.
\end{proof}

\begin{remark}
	It is clear that the developing map associated preserves the structure and, in particular, the induced action on the local model $G'\slash H'$ is given by
	\begin{align*}
	\R^k\times G'\slash H' &\to G'\slash H'\\
	(t_1,\dots,t_k, gH')&\mapsto gH'\exp(t_1\hat X_1 + \dots + t_k \hat X_k)
	\end{align*}
	
Finally, should $\Omega$ be equal to $M$, and the developing map $\tilde M \to G'\slash H'$ be a diffeomorphism, then, the canonical representation $\rho: \pi_1(M)\to G'$ give a diffeomorphism
	
$$M =  \pi_1(M)\backslash \tilde M \to \rho(\pi_1(M))\backslash G'\slash H'$$
which smoothly conjugates our original action with an quasi-algebraic one.
\end{remark}

        \section{Extending the structure}\label{extendsec}
           The goal of this section is to extend the $(G',\hat V)$ structure of $\Omega$ to $M$. First some technical preparations.

We can build a connection $\hat \nabla$ on $\hat V$ in a similar to the one on $M$. It is clear that $G'$ is a group of affine transformation with respect to $\hat \nabla$. This connection is also compatible with the $A$-structures on $M$ and $\hat V$, which allow us to prove the following lemma:
\begin{lemma}\label{extge}
	The geodesics of $\hat\nabla$ tangents to the distributions $\hat E^+$ are complete
\end{lemma}

\begin{proof}
    First, let us define the set
    $$\Delta = \{p\in M|F^+(p)\subset\Omega \;\;and\;\;F^-(p)\subset\Omega\}$$
    \begin{description}
    \item[Claim:] $\Delta$ is dense.
    
    In fact, first we notice that compact orbits are dense in $M$ (they are dense in the non wandering set, which we have proved to be all $M$). As $\Omega$ is open and dense, it follows that compact orbits are dense in the entire $\Omega$.
	
	We shall prove that $\Delta$ contains every compact orbit in $\Omega$.  It is known that the Anosov foliations depends only on the open connected cone of Anosov elements we chose. Let $p\in \Omega$ with a compact orbit. We can take an Anosov element $X'$ with flow $\varphi_t'$ in the same cone as $X$ such that $\varphi_{t_0}'(p) = p$. 
	
	Now, let $q\in F^+(p)$, then
	\begin{align*}
	0 =\lim_{t\to +\infty}d(\varphi_t'(q),\varphi_t'(p)) &= \lim_{n\to \mp\infty}d(\varphi_{nt_0}'(q),\varphi_{nt_0}'(p))\\
	&= \lim_{n\to \mp\infty}d(\varphi_{nt_0}'(q),p)
	\end{align*}
	thus $\varphi_{nt_0}'(q) \to p$. As $\Omega$ is open, and invariant by the action, it follows that $q\in \Omega$. We have proved that $F^+(p)\subset\Omega$. A similar argument proves the case $F^-(p)$.

    \end{description}

	Consider $\hat Y \in \hat E^+$ with base-point $\hat p\in \hat V$. Lets prove that a geodesic with initial condition $\hat Y$, integrates to up to time greater then $1$. There exists a connected, simple connected open set $O\subset\Omega$, and a developing map $\theta:O\to\hat V$ and $Y\in E^+_{v}$ such that, $\theta_*Y=\hat Y$. 
	From the density of $\Delta$ and the transitivity of the pseudo-group $Aut^{loc}(\sigma)$, we can move the base point $v$ of $Y$ a little bit and suppose that $v\in \Delta$. 
	Now, we have already proved that the geodesics in $\Omega$ tangents to $E^+$ are complete, thus, they integrate up to a time greater the one. In particular, the geodesic $t\to \gamma(t)$, starting from $v$ with direction $Y$ integrate up to time greater the one. 
	As $\nabla E^+\subset E^+$, we have $\dot\gamma\in E^+$,
	thus, $\gamma$ never leaves the leaf $F^+(v)$. As $v\in \Delta$, $\gamma(t)\in \Omega$, and we can suppose that $\gamma([0,1])\subset O$\footnote{The developing map is defined for any simple connected open set, We just take $O'$ to be a tubular neighbourhood of $\gamma$, and $\theta'$ the developing map of $O'$ that coincides with $\theta$ on the point v }. The curve $t\mapsto \theta \circ\gamma$ is our desired geodesic.
\end{proof}

\begin{remark}
	At first we may be tempted to just use the completeness of the geodesics in $M$ (Lemma \ref{lemageod}). But we build our $(G',\overline V)$-structure on $\Omega$, which means that the developing maps are defined on the universal cover $\tilde \Omega$, and not on $\tilde M$. That is, maybe our geodesic passes through a point in $M\backslash \Omega$, and in this case, we won't be able to develop our geodesic on $\overline V$.
	
	Our next section will solve this problem by extending the $(G',\overline V)-$structure to all $M$.
\end{remark}

Now,let us define appropriated local coordinates for $M$ using the completeness of the connection along the strong leaves.

Let $\varphi^j_t$, $j=1,\dots,k$ be the flows of the vector fields $X_j$. Define, for every $v\in M$ the map
\begin{align*}
\Psi_v:T_vM &= E_v^+ \oplus\bigoplus_j\R X_j\oplus E_v^-  \to M&\\
Y&=Y^+ + t_1X_1 +\dots +  t_kX_k + Y^-\mapsto \varphi^1_{t_1}\circ\dots\circ\varphi^k_{t_k}(\exp^{\nabla}(\tau_{Y^+}(Y^-)))
\end{align*}
where $\exp^{\nabla}$ is the exponential with respect to the connection $\nabla$ and $\tau_{Y^+}(Y^-)$ is the parallel transport for time $t =1$ along the geodesic $t\mapsto \exp^{\nabla}(tY^+)$ of the vector $Y^-$. 

We define, in an analogous way, for every $\tilde v\in\tilde M$ and $\hat v\in \hat V$ the maps  $\tilde\Psi_:T_{\tilde v}\tilde M\to \tilde M$ and $\hat\Psi_:T_{\hat v}\hat V\to \hat V$.

From Lemas \ref{extge} and \ref{lemageod}, we see that those maps can be defined even for arbitrary large $Y^+$,$\tilde Y^+$ and $\hat Y^+$.

\begin{definition}
	An open set $O\subset M$ will be called $A$-star-shaped with respect to $v\in O$ if there exists an open set $U\subset T_vM$ such that
	\begin{itemize}
		\item{}$\Psi_v$ is an diffeomorphism from $U$ to $O$. 
		\item{}If $Y = Y^+ + t_1X_1+\dots + t_kX_k + Y^-$ is in $U$, then, for every $s\in [0,1]$, so are
		\begin{itemize}
			\item{}$Y^+ + 
					t_1X_1 +\dots +  			t_kX_k +	   s Y^-$
			\item{}$s Y^+ + t_1X_1 +\dots + 			t_kX_k $
			\item{}$t_1X_1 +\dots +s t_jX_j+\dots + t_kX_k$ for every $j\in \{1,\dots,k\}$
		\end{itemize}
	\end{itemize}
	In particular, $O$ is contractible and therefore simple connected.
\end{definition}

We have defined the $A$-star-shaped open sets in such way that they are well behaved with respect to the affine transformations, that is, if $O$ is an $A$-star-shaped open set with respect to $v$ and $\theta$ is a developing map  from $O$ to $\hat V$, then
\begin{align}\label{Aetoiles}
\theta\circ\Psi_v = \hat\Psi_{\theta(v)}\circ T_v\theta
\end{align}

\begin{lemma}\label{starshape}
	Let $v \in \Delta$ and $O$ be an $A$-star-shaped open set with respect to $v$. Then, there exists an open dense subset $O^\prime\subset O\cap\Omega$ which is also $A$-star-shaped with respect to $v$.
\end{lemma}

\begin{proof}
	We consider the open set $U\subset T_vM$ associated with the $A$-star-shaped set $O$ and define
	\begin{align*}
	U':=\{Y = Y^++t_1X_1+&\dots+t_kX_k+Y^-\in U\;|\\
	&\Psi_v\left(  Y^++t_1X_1+\dots+t_kX_k+s Y^- \right)\in \Omega \;\forall s\in [0,1]\}
	\end{align*}
	
	We take $O'=\Psi_v(U')$. It is clearly an open set. Let us see that it is also dense.\\

	For any open set $W\subset \Omega\cap O$ take $p\in W\cap\Delta$. We write $p = \Psi_v(Y_0^+ +t_{01}X_1 + \dots + t_{0k}X_k + Y_0^-)$.We must show that $\Psi_v(Y_0^+ +t_{01}X_1 + \dots + t_{0k}X_k + sY_0^-)\in \Omega$ for every $s\in[0,1]$.\\
	
	Remember that, by definition, $\Psi_v$ is defined by travelling a certain distance along the unstable foliation, then travelling a certain distance along the stable foliation, and acting a certain vector on the result.
	
	As $\Delta$ is invariant by the action, we have that $\Psi_v(Y_0^+ + Y_0^-)\in\Delta$.\\
	
	By definition, the stable foliation of $\Psi_v(Y_0^+ + Y_0^-)$ is contained in $\Omega$, in particular, $\Psi_v(Y_0^+ + sY_0^-)\in \Omega$. But $\Omega$ is invariant by the action, and thus
	$$\Psi_v(Y_0^+ +t_{01}X_1 + \dots + t_{0k}X_k+ sY_0^-)\in \Omega$$

	It remains to show that $O'$ is $A$-star-shaped. 
	
	As $\Omega$ is invariant by the action and $v\in\Delta\subset\Omega$, we have
	$$t_1X_1+\dots+t_kX_k\in U' \Leftrightarrow \Psi_v( +t_1X_1+\dots+s t_jX_j+\dots+t_kX_k) \in \Omega$$
	that is
	$$t_1X_1+\dots+t_kX_k\in U' \Leftrightarrow  +t_1X_1+\dots+s t_jX_j+\dots+t_kX_k \in U'$$
	
	We must now show, that for any $s\in[0,1]$ we have
	$$Y^++t_1X_1+\dots+t_kX_k\in U' \Leftrightarrow s Y^++t_1X_1+\dots+t_kX_k \in U'$$
	that is
	$$Y^++t_1X_1+\dots+t_kX_k\in U' \Leftrightarrow \Psi_v\left(s Y^++t_1X_1+\dots+t_kX_k\right) \in \Omega$$

But $v\in\Delta$ and thus, it's unstable foliation is in $\Omega$. In particular, $\Psi_v(sY^+)\in\Omega$ for every $s$. As $\Omega$ is invariant by the action, $\Psi_v\left(s Y^++t_1X_1+\dots+t_kX_k\right) \in \Omega$

\end{proof}	
The following lemma is a simple variation of the usual result about the existence of normal neighborhoods of uniform sizes.
\begin{lemma}
	Fix a background metric $g$ on $M$ and let us denote, for $v\in M$,$r\in\R_{>0}$ and $Y = Y^++t_1X_1+\dots+t_kX_k+Y^- \in T_{v}M$:
	$$
	B^r(v) = \{Y\in T_vM\;;\;\|Y\|^2\stackrel{def}{=} g(Y^+,Y^+)+|t_1|^2+\dots+|t_k|^2+g(Y^-,Y^-)<r^2\}$$
	and 
	$$B_r(v) = \{p\in M;\; dist_g(c,p)<r\}$$
	Then, there exists $\delta,\varepsilon>0$ such that for every $v\in M$, $\Psi_v:B^\delta(v)\to M$
	is a diffeomorphism onto it's image and
	$B_\epsilon(v)\subset \Psi_v(B^\delta(v))$
\end{lemma}

\begin{coro}\label{hauhua}
	For any dense subset $U\subset M$, we can cover $M$ by $A$-star shaped open sets with base points in $U$.
\end{coro}

\begin{lemma}
	$\Omega=M$
\end{lemma}

\begin{proof}
	We just need to prove that the $(G',\hat V)$-structure on $\Omega$ can be extended to $M$.
	
	Using the Corolary \ref{hauhua}, we take a cover $\{O_{v_i}\}$ of $M$ by $A$-star shaped open sets with basepoints $v_i\in\Delta$, and using Lemma \ref{starshape}we consider open, subsets $O_{v_i}'\subset \Omega\cap O_{v_i}(\varepsilon)$ dense in $O_{v_i}$ and also $A$-star-shaped with respect to $v_i$.

	In particular, $O_{v_i}'$ is contractible, and therefore, simple connected. there exists, thus, a developing map $\theta:O_{v_i}'\to \hat V$.  We define $\hat{\theta}: O_{v_i}\to\hat V$ by
	\begin{align*}
	\hat\theta = \hat{\Psi_{\theta(v_i)}}\circ(T_{v_i}\theta)\circ\Psi_{v_i}^{-1}
	\end{align*}
	
	It follows from \ref{Aetoiles} that $\hat{\theta}$ is an extension of $\theta$. From the density of $O_{v_i}'$ in $O_{v_i}$, it follows that $\hat{\theta}_*(X_{j})=\hat X_{j}$, $\hat{\theta}_*(E^\pm) = \hat E^\pm$ and $\hat{\theta}^*(\omega) = d\lambda$. We conclude that $\hat{\theta}$ is a developing map\footnote{A developing map is a local diffeomorphism ($\hat{\theta}$ is the composite of local diffeomorphisms) that preserves the invariant bundles and the $2$-form.}. That is, we managed to extend the $(G',\hat V)-$structure to $M$.
\end{proof}
        \section{Completeness}\label{completesec}
           Finally we shall prove that the extended structure build in the previous section is complete, that is
\begin{proposition}\label{complet}
	Let $\tilde M$ be the universal cover of $M$ and  let $\theta:\tilde M \to \hat V$ be a developing map of the $(G',\hat V)$-structure. Then $\theta$ is a covering map..
\end{proposition}
\begin{proof}
	To prove our proposition, we must find, for each $w\in \hat V$ a neighbourhood $O$ such that $\theta^{-1}(O)$ is the disjoint union of open sets $\{O_i\}_{i\in I}$ such that the restriction of $\theta$ to each $O_i$ is a diffeomorphism onto $O$.
	
	Let $w\in\hat V$ and ${v_i}_{i\in I} = \theta^{-1}(w)$. Consider $O$ an $A$-star-shaped open set with respect to $w$ and $U\subset T_w\hat V$ the associated open set.
	
	We define, for $i\in I$
	\begin{align*}
	O_i = \tilde{\Psi}_{v_i}\big((T_{v_i}\theta)^{-1}(U)\big)
	\end{align*} 
	
	It follows from \ref{Aetoiles} that $\theta$ induces an diffeomorphism $\theta_i$ from $O_i$ to $O$.
	
	Lets show that $O_i\cap O_j = \emptyset$ if $i\neq j$.  For this, we consider the set
	\begin{align*}
	O' = \{w'\in O\;|\:\theta_i^{-1}(w') = \theta_j^{-1}(w')\}
	\end{align*}
	
	The set $O'$ is obviously closed\footnote{$O'$ is the intersection of two closed sets, the inverse images by continuous functions of a point}. 
	
	On the other hand, $O'$ is also open, for $\theta_i$ and $\theta_j$ are the restriction of a $\theta$ to a certain open sets. If there is a intersection (which is the case if $\theta_i^{-1}(w') = \theta_j^{-1}(w')$ for some $w'$), then $\theta_i = \theta_j$ on this intersection $V'$, and thus $\theta(V')\subset O'$.

	 As $w\not\in O'$ if $i\neq j$ and $O$ is connected (actually contractible), it follows that $O'=\emptyset$. 
	
	Finally, we shall show that $\theta^{-1}(O)=\cup_{i\in I}O_i$
	
	First we define the map
	\begin{align*}
	\tilde\Phi:T\tilde M&\to T\tilde M\\
	\tilde Y &= \tilde Y^+ + t_1\tilde X_1+\dots +t_k\tilde X_k+ \tilde Y^-&\mapsto (\tilde\varphi^1_{t_1})_*\dots(\tilde \varphi^k_{t_k})_*\big(\tilde \tau_{(\tilde \tau_{\tilde Y^+}\tilde Y^-)}(\tilde \tau_{\tilde Y^+} \tilde Y)\big)
	\end{align*}
	 and in a similar way the map $\hat \Phi: T\hat V\to T\hat V$. It is clear that
	 
	 $$T\theta\circ\tilde \Phi = \hat{\Phi}\circ T\theta$$

	Moreover, $\tilde\Phi$ is a diffeomorphism\footnote{in fact, we shall write explicit formulas for its inverse. Consider the smooth maps $\tilde f^\pm, \tilde g^\pm: T\tilde M \to T\tilde M$ defined by
	\begin{align*}
	\tilde f^\pm(\tilde Y) &=  \tilde \tau_{\tilde Y^\pm}\tilde Y\\
	\tilde g^\pm(\tilde Y) &=  \tilde \tau_{-\tilde Y^\pm}\tilde Y
	\end{align*}
	It is clear that $\tilde f^\pm\circ \tilde g^\pm = \tilde g^\pm\circ \tilde f^\pm = Id$. which makes $\tilde f^\pm$ diffeomorphisms. We can write 
	$$\tilde\Phi(\tilde Y) = (\tilde\varphi^1_{t_1})_*\dots(\tilde \varphi^k_{t_k})_*(\tilde f^-\circ\tilde f^+(\tilde Y))$$.
	Let $\tilde Z = \tilde Z^+ + t_1X_1+\dots +  t_kX_k + \tilde Z^-$ and let 
	$$(\tilde\varphi^1_{-t_1})_*\dots(\tilde \varphi^k_{-t_k})_*(\tilde Z):=\tilde W $$ 
	The inverse $\tilde\Phi^{-1}$ is given by
		$$\tilde\Phi^{-1}(\tilde Z) = g^+\circ g^- 	(\tilde W)$$
		
		The key element here is the fact that if we write $$\tilde \Phi(\tilde Y^+ +t_1\tilde X_1+\dots+t_k\tilde X_k + \tilde Y^- ) = \tilde Z^+ + s_1\tilde X^1+\dots +s_k\tilde X_k + \tilde Z^-$$, then $s_j = t_j$ for every $j$.}

	Now, consider $v'\in \hat V$ such that $\theta(v')\in O$. We can write $\theta(v') = \tilde \Psi_w(\hat Y)$ for $\hat Y\in U$. Consider
	 $$\tilde Y := (\tilde \Phi)^{-1}\big((T_{v'}\theta)^{-1}(\hat \Phi(\hat Y))\big)$$
	and let $v$ be the basepoint of $\tilde Y$. We have
	\begin{align*}
	T\theta(\tilde Y)&= \hat{\Phi}^{-1}\circ T\theta\circ\tilde \Phi(\tilde Y)\\
	&= \hat{\Phi}^{-1}\circ T\theta (T_{v'}\theta)^{-1}(\hat \Phi(\hat Y)) = \hat{\Phi}^{-1}(\hat \Phi(\hat Y)) = \hat Y
	\end{align*}
	Thus the base point $v$ must be $v_i$ for some $i\in I$. It remains to prove that 
	$$\tilde \Psi_{v_i}(\tilde Y) =v'$$
	This will conclude the proof, for we have shown that $v'\in O_i$.
\end{proof}

\begin{lemma}
	If $\tilde Z$ has basepoint $a$, and $\Phi^{-1}(\tilde Z)$ has basepoint $b$, then $\tilde \Psi_b\circ\tilde \Phi(\tilde Z) = a$.
\end{lemma}
\begin{proof}
	Notice that if we take an element $T_b\tilde M\ni \tilde Y = \tilde Y^+ + t_1\tilde X_1 + \dots + t_k\tilde X_k + \tilde Y^-$, then we can understand
	the map $\Psi_b$ in the following way.
	
	We take the geodesic with starting vector $\tilde Y^+$ and transport along it, to time one, the vector $\tilde Y^-$. Now we take a geodesic with this starting vector and take it's time one and apply $\varphi^1_{t_1}\circ\dots\circ\varphi^k_{t_k}$.
	
	We can understand the map $\Phi$ in the following way.
	We take the geodesic with starting vector $\tilde Y^+$ and transport along it, to time one, both the vector $\tilde Y^-$. We now take a geodesic with this starting vector and take it's time one. We transport the original vector $\tilde Y$ along both of this segments, and apply $D(\varphi^1_{t_1}\circ\dots\circ\varphi^k_{t_k})$ o the result. It is clear that $\Phi(\tilde Y)$ will
	have basepoint $\Psi(Y)$. We take $\tilde Y = \Phi^{-1}(\tilde Z)$ and the result follows.
	
\end{proof}

\bibliographystyle{amsplain}


%
%
%
%
\end{document}